\pdfoutput=1 

\documentclass[a4paper,oneside,reqno,10pt,final]{article}

\usepackage{a4wide}
\usepackage[utf8]{inputenc}
\usepackage[a4paper]{geometry}

\usepackage{titlesec}

\usepackage{bm}

\usepackage{enumitem}
\setlist{itemsep=3pt,parsep=1pt,topsep=3pt,partopsep=1pt}
\setlist[enumerate,1]{label=(\alph*)}
\setlist[enumerate,2]{label=(\roman*)}

\usepackage[usenames,dvipsnames]{xcolor}

\usepackage{url}
\usepackage[colorlinks=true,linkcolor=blue,citecolor=Violet,linktocpage=false]{hyperref}
\usepackage{graphicx}

\usepackage{mathtools}
\mathtoolsset{showonlyrefs}

\usepackage[only,llbracket,rrbracket]{stmaryrd}

\usepackage[T1]{fontenc}

\usepackage[p,osf]{scholax}
\usepackage{amsmath,amsthm}
\usepackage[scaled=1.075,ncf,vvarbb]{newtxmath}

\usepackage{chngcntr}

\usepackage[usenames,dvipsnames]{xcolor}

\swapnumbers

\theoremstyle{plain}
\newtheorem{theorem}{Theorem}[section]
\newtheorem{lemma}[theorem]{Lemma}

\newtheorem*{theorem*}{Theorem}
\newtheorem*{corollary*}{Corollary}

\theoremstyle{definition}

\newtheorem{miniremark}[theorem]{}

\newtheorem*{notation*}{Notation}

\theoremstyle{remark}

\swapnumbers

\newtheoremstyle{newremark}  
{}                           
{}                           
{}                           
{}                           
{}                           
{.}                          
{ }                          
{{\bfseries \thmnumber{#2}}{\itshape \thmname{ #1}}\thmnote{ (#3)}} 

\theoremstyle{newremark}
\newtheorem{remark}[theorem]{Remark}

\newcommand{\adim}{n+1}
\newcommand{\vdim}{n}

\newcommand{\oball}[2]{\mathbf{U}(#1,#2)}

\newcommand{\sphere}[1]{\mathbb{S}^{#1}}

\ifdefined\textint
    \renewcommand{\textint}[2]{{\textstyle\int_{#1}^{#2}}}
\else    
    \newcommand{\textint}[2]{{\textstyle\int_{#1}^{#2}}}
\fi

\ifdefined\textfint
    \renewcommand{\textfint}[2]{{\textstyle\fint_{#1}^{#2}}}
\else    
    \newcommand{\textfint}[2]{{\textstyle\fint_{#1}^{#2}}}
\fi

\ifdefined\textsum
    \renewcommand{\textsum}[2]{{\textstyle\sum_{#1}^{#2}}}
\else  
    \newcommand{\textsum}[2]{{\textstyle\sum_{#1}^{#2}}}
\fi

\ifdefined\textprod
  \renewcommand{\textprod}[2]{{\textstyle\prod_{#1}^{#2}}}
\else  
  \newcommand{\textprod}[2]{{\textstyle\prod_{#1}^{#2}}}
\fi

\newcommand{\R}{\mathbf{R}}

\newcommand{\LM}{\mathscr{L}}

\DeclareMathOperator{\dist}{dist}

\newcommand{\distF}[2]{\textrm{dist}_{\!\!#2}^{#1}}
\newcommand{\da}[2]{\distF{#2}{#1}}

\newcommand{\an}[2]{\boldsymbol{\nu}_{\!#1}^{#2}}

\newcommand{\ar}[2]{\boldsymbol{r}_{\!#1}^{#2}}

\newcommand{\anp}[2]{\boldsymbol{\xi}_{#1}^{#2}}

\newcommand{\pp}{\mathbf{p}}

\newcommand{\HM}{\mathscr{H}}

\newcommand{\restrict}{ \mathop{ \rule[1pt]{.5pt}{6pt} \rule[1pt]{4pt}{0.5pt} }\nolimits }

\newcommand{\ud}{\ensuremath{\,\mathrm{d}}}

\newcommand{\uD}{\ensuremath{\mathrm{D}}}
\newcommand{\Der}{\uD}

\DeclareRobustCommand{\rchi}{{\mathpalette\irchi\relax}}
\newcommand{\irchi}[2]{\raisebox{\depth}{$#1\chi$}}

\newcommand{\id}[1]{\bm{1}_{#1}}

\newcommand{\lIm}{ [ }
\newcommand{\rIm}{ ] }

\newcommand{\VF}{\mathscr{X}}

\newcommand{\Bdry}{\partial}

\newcommand{\redbd}{\partial^{\ast}}

\DeclareMathOperator{\ap}{ap}

\DeclareMathOperator{\lin}{span}

\newcommand{\cnt}[1]{\mathscr{C}^{#1}}

\newcommand{\polar}[1]{{#1}^{\circ}}
\newcommand{\ppolar}[1]{{#1}^{\circ\circ}}

\newcommand{\grad}{\nabla}

\DeclareMathOperator{\spt}{spt}

\DeclareMathOperator{\Tan}{Tan}

\DeclareMathOperator{\Nor}{Nor}

\DeclareMathOperator{\Lip}{Lip}

\DeclareMathOperator{\diam}{diam}

\newcommand{\without}{\mathbin{\raisebox{0.2ex}{$\smallsetminus$}}}

\newcommand{\wulff}{\mathcal{W}}

\newcommand{\perim}[1]{\mathcal{P}_{#1}}

\DeclareMathOperator{\im}{im}

\newcommand{\symdiff}{\mathbin{\triangle}}

\DeclareMathOperator{\Cut}{Cut}

 \allowdisplaybreaks

\author{Mario Santilli}

\title{Bubbling of almost critical points \\
  of anisotropic isoperimetric problems \\
  with degenerating ellipticity}

\begin{document}

\maketitle

\begin{abstract}
Given a sequence of uniformly convex norms $ \phi_h $ on $ \R^{n+1} $ converging  to an arbitrary norm $ \phi $, 	we prove  rigidity of $ L^1 $-accumulation points of sequences of sets $ E_h \subseteq \R^{n+1} $ of finite perimeter, that are volume-constrained almost-critical points of the anisotropic surface energy functionals associated with $ \phi_h $.  Here, almost criticality is measured in terms of the $ L^n $-deviation from being constant of the distributional anisotropic mean $ \phi_h $-curvature of (the varifold associated to) of the reduced boundaries of $ E_h $. We prove that such limits are finite union of disjoint, but possibly mutually tangent, $ \phi $-Wulff shapes.

\end{abstract}

\section{Introduction}

The Wulff theorem (cf.\ \cite{Dinghas}, \cite{taylor78}, \cite{fonsecawulff}, \cite{fonsecamuellerwulff}, \cite{brothersmorgan}) uniquely characterizes volume-constrained minimizers of anisotropic surface energies. Given a norm $ \phi $ on $ \R^{n+1}  $, one defines the anisotropic surface $ \phi $-energy (or $ \phi $-perimeter) $ \perim{\phi}(E)  $ of a set of finite perimeter $ E \subseteq \R^{n+1} $ by the formula
$$ 	\perim{\phi}(E)= \textint{\redbd E}{} \phi(\an{E}{}(x)) \ud \HM^{\vdim}(x)\,, $$
where $ \redbd E $ is the reduced boundary and  $ \an{E}{} : \partial^\ast E \rightarrow \mathbb{S}^n $ is the measure-theoretic
exterior normal of $ E $, and the Wulff theorem ensures that the compact convex body
$$ \wulff^\phi = \bigcap_{\nu \in \sphere{n}}\bigl\{x \in \R^{n+1} : x \bullet \nu < \phi(\nu)\bigr\}  $$
is the unique minimizers of  the anisotropic surface $ \phi $-energy among all sets of finite perimeter $ E \subseteq \R^{n+1} $ with the same volume of $ \wulff^\phi $. If $ \phi $ is the Euclidean norm, then $ \wulff^\phi $ is the round ball radius $ 1 $, while crystalline norms are those for which $ \wulff^\phi $ is a convex polyhedron. In general, we call the rescaled set $ r \wulff^\phi = \wulff^\phi_r $, \emph{$ \phi $-Wulff shape (or radius $r$)}.

Suppose $ \phi $ is a  uniformly convex $ \cnt{2} $-norm. For every set of finite perimeter $ E \subseteq \R^{n+1} $ the first variation of $ \mathcal{P}_{\phi} $ at $ E $ is a linear functional $ \delta \perim{\phi}(E) : \VF(\R^{n+1}) \rightarrow \R $ (we denote by $ \VF(\R^{n+1}) $ the space of smooth vector fields with compact support), 
and we say that a bounded set $ E $ is a volume-constrained critical point of $ \perim{\phi} $ if $$ \delta \perim{\phi}(E)(g) = 0 $$ for every vector field $ g \in \VF(\R^{n+1}) $ whose associated integral flow $ \varphi_t $ is volume-preserving.   
Volume-constrained critical points of $ \perim{\phi} $ can be equivalently characterized by the requirement that there exists $ \lambda > 0 $ such that 
\begin{equation}\label{intro first variation}
\delta \perim{\phi}(E) (g) = \lambda \textint{\partial^\ast E}{} g(x) \bullet \an{E}{}(x)\, \ud \HM^n(x) \quad \text{for $ g \in \VF(\R^{n+1})\,;$} 
\end{equation} 
	indeed, $ \lambda = \tfrac{n\perim{\phi}(E)}{(n+1)\,|E|} $.  If $ E $ is an open set with smooth boundary and a volume-constrained critical point for $ \perim{\phi} $, then $ \partial E $ has constant mean $ \phi $-curvature (when $ \phi $ is the Euclidean norm then the mean $ \phi $-curvature coincides with the classical mean curvature); cf.\ \ref{curvature smooth hypersurfaces}.  It is proved in \cite{helimage} that  a compact embedded smooth hypersurface with constant mean $ \phi $-curvature is the boundary of a Wulff shape. This result generalizes the celebrated Alexandrov's soap bubble theorem to the anisotropic setting and proves uniqueness of \emph{smooth} critical points of volume-constrained critical points of $ \perim{\phi} $. The sutler problem of proving that arbitrary volume-constrained critical points of $ \perim{\phi} $ are uniquely characterized as finite unions of disjoint and possibly mutually tangent Wulff shapes of the same radius is studied in \cite{delgadinomaggi} when $ \phi $ is the Euclidean norm, and in \cite{deRosaKolaSantilli} for uniformly convex smooth norms.   
	
	It should also be mentioned that in $ \R^2 $ a complete characterization of volume constrained critical points of $ \perim{\phi}  $ for arbitrary norms $ \phi $ is proved by Morgan in \cite{morganwulff}.

Starting with the pioneering works of \cite{struwe84} and \cite{breziscoronRELL}, a lot of effort has been devoted to study  hypersurfaces with almost constant mean curvature. A prominent area of research in this direction is represented by the study of bubbling phenomena for sequences of sets with almost constant mean curvature. Roughly speaking, in the Euclidean setting this problem amounts to understand whether a small deviation (in some $ L^p $-norm) of the Euclidean mean curvature of a smooth set $ E $ from being constant implies proximity of $ E $ to a finite union of disjointed and possibly mutually tangent balls; cf.\  \cite{ciraolomaggi2017},  \cite{JulinNiinikoski}, \cite{poggesibubbling} and \cite{JulinMoriniOronzioSpadaro} for a non exhaustive list of papers on the subject.   While these problems have been extensively investigated in the Euclidean setting, much less appears to be known about the following more general bubbling problem (BP) for arbitrary anisotropies.

\begin{equation}
	\tag{\rm BP}\label{BP}
	\begin{split}
&\mbox{ \it	If $ \phi_h $ are uniformly convex smooth norms in  $\R^{n+1} $ converging  to an arbitrary norm $ \phi $,} \\
&\mbox{\it are finite unions of disjointed and possibly mutually tangent $ \phi $-Wulff shapes} \\
& \mbox{\it  the only $ L^1 $-accumulation points of sequences $ \Omega_h $ of bounded  open sets in $ \R^{n+1} $ }\\ 
&\mbox{\it whose mean $ \phi_h $-curvatures  $ L^p $-converge (for some $ p \geq 1 $) to a constant?}
	\end{split}
\end{equation}

\noindent To author's knowledge, this question is studied for the first in \cite{delgadinomaggimihailaneumayer}, and serves as a weak version of the Alexandrov soap bubble problem for arbitrary anisotropies. In particular, \cite[Theorem 1.4]{delgadinomaggimihailaneumayer}  partially answers \eqref{BP} for sequences of smooth sets $ \Omega_h $ whose (scalar)  mean $ \phi_h $-curvatures $ H^{\phi_h}_{\Omega_h} $ converge in $ L^2 $ to a constant, under the following additional hypothesis:  (i) the mean $ \phi_h $-curvatures of $ \Omega_h $ are uniformly bounded from below by a positive constant (i.e.\ uniform mean convexity), and (ii) the sequence of $ L^2 $-deviations $ \| H^{\phi_h}_{\Omega_h} - \lambda \|_{L^2(\partial \Omega_h)}  $ vanishes faster than the rate of degeneracy of the ellipticity constants of $ \phi_h $. As mentioned in \cite[pag.\ 1141]{delgadinomaggimihailaneumayer}, the fast convergence condition required in (ii) is a key technical assumption for the method used in \cite{delgadinomaggimihailaneumayer}.

In this paper, employing a completely different approach, we fully settle this question for sequences of sets in $ \R^{n+1} $ whose mean $ \phi_h $-curvatures $ L^n $-converge to a constant. (For $ n = 2 $,  that is in $\R^3 $, our result directly generalizes \cite[Theorem 1.4]{delgadinomaggimihailaneumayer}, in particular removing the fast-convergence condition.) Notably, our main result allows for non-smooth approximating sequences  $ \Omega_h $ in \eqref{BP}, more precisely sets of finite perimeter $ \Omega_h $ whose reduced boundary admits a distributional mean $ \phi_h $-curvature in the sense of varifolds; see below for a precise statement. The latter is a distinctive feature of our result, not merely motivated by a quest for generalization, but driven by the concrete fact that it is often not possible to obtain approximating sequences of \emph{smooth} hypersurfaces as solutions of higher-dimensional geometric variational problems (e.g.\ almost-minimizers of elliptic geometric integrands). Under this point of view, our result is new even in the Euclidean setting, cf.\ Remark \ref{rmk compactness 3}, and provides a far-reaching generalization of the main result in \cite{deRosaKolaSantilli} to arbitrary anisotropies.

We state now the main result of this paper, specifying first the hypothesis on the approximating sequence. We consider a sequence of uniformly convex $ \cnt{3} $-norms $ \phi_h $ and a sequence of sets of finite perimeter $ E_h \subseteq \R^{n+1} $  such that
$$ 	\HM^n(\overline{\partial^\ast E_h} \setminus \partial^\ast E_h) =0 \quad \text{for every $ h $} $$ 
and the first variation $ \delta \perim{\phi_h}(E_h) $ is representable by integration in $ L^\infty $, namely 
	$$ \delta \perim{\phi_h}(E_h)(g) = \int_{\partial^\ast E} g(x) \bullet \overrightarrow{H}^{\phi_h}_{E_h}(x)\, d\HM^n(x) \; \text{for $ g \in \VF(\R^{n+1}) $ and $ \overrightarrow{H}^{\phi_h}_{E_h} \in L^\infty(\HM^n \restrict \partial^\ast E_h) $.} $$
Set $H^{\phi_h}_{E_h}= \overrightarrow{H}^{\phi_h}_{E_h} \bullet \an{E_h}{}  $, where $ \an{E_h}{} $ is the exterior unit-normal of $ E_h $.
\begin{theorem}[\protect{cf.\ Theorem \ref{thm compactness}}]
	Suppose the norms  $\phi_h$  pointwise converge to an arbitrary norm $ \phi $ and $ \{E_h\}_{h \geq 1}  $ is a sequence of sets of finite perimeter of $ \R^{n+1} $ as above and such that
	\begin{equation*}
		\sup_{h \geq 1} \bigl( \diam (E_h) + \HM^{\vdim}(\partial^\ast E_h)\bigr) < \infty 
	\end{equation*}
and
\begin{equation*}
		\| H^{\phi_h}_{E_h} - \lambda \|_{L^{\vdim}(\partial^\ast E_h)} \rightarrow 0 \quad \textrm{as $ h \to \infty $.} 
	\end{equation*}
	
	Then the sets $ E_h $  converge in measure to a finite union of disjointed and possibly mutually tangent Wulff shapes. 
\end{theorem}

While \cite[Theorem 1.4]{delgadinomaggimihailaneumayer} is based on a potential theoretic proof done in the spirit of \cite{ros87ibero}, here we follow a completely different route based on an integral-geometric approach that generalizes the methods employed in \cite{montielros},  \cite{delgadinomaggi}, \cite{HugSantilli} and \cite{JulinNiinikoski}. In particular, the key estimates contained in Lemma \ref{lemma estimates} provide a subtle geometric-measure theoretic generalization of \cite[Proposition 3.3]{JulinNiinikoski} to a singular anisotropic setting. Indeed, a major challenge in this lemma arises from the fact that these estimates are proved for finite perimeter sets, rather than $\cnt{2} $-regular sets. This means that the euclidean differential-geometric approach employed in \cite{JulinNiinikoski} cannot be used, and needs to be carefully redesigned in the anisotropic non-smooth setting of Lemma \ref{lemma estimates}, that is for varifolds with bounded distributional anisotropic mean curvature.  This is a particularly delicate issue, since the structural theory needed to treat such varifolds is much weaker than its Euclidean counterpart (cf.\ \cite{Allard1986}, \cite{KolSan25} and \cite{KolSan26}). 

Finally, we mention that question \eqref{BP} is naturally related with the study of anisotropic volume preserving flows. For convex initial data and arbitrary anisotropies $ \phi $, the existence of solutions and convergence to $ \phi $-Wulff shapes was obtained in \cite{BellettiniCasellesChambolleNovaga}. For certain non-convex initial data, the existence of solutions is studied in \cite{KimKwonPovzar}, while the question of convergence of the flow to $ \phi $-Wulff shapes remains open (cf.\ \cite[Remark 1.1.3]{KimKwonPovzar}). In view of the successful applications of bubbling theorems to euclidean volume preserving flows (cf.\ \cite{JulinNiinikoski}), a result like Theorem \ref{thm compactness} might be beneficial in the general anisotropic setting. 

\section{Preliminaries}

\subsection*{Basic background}
If $ A , B $ are subsets of a vector space we write $ A + B = \{a + b : a \in A, \; b \in B\} $.  If $ A \subseteq \R^{n+1} $ we denote its $ (n+1) $-dimensional Lebesgue measure by $ | A | $. If $ T \subseteq \R^{n+1} \times \R^{n+1} $ and $ A \subseteq \R^{n+1} $ we define 
$$ T| A =  \{(a, u) \in T : a \in A\}. $$

\begin{miniremark}\label{basic wulff shapes}
    Suppose $ \phi $ is a norm on $ \R^{\adim} $. The \emph{dual} (or \emph{polar}) norm
    $ \polar{\phi} $ is defined as
    \begin{displaymath}
        \polar{\phi}(u) = \sup \bigl\{
        u \bullet v : v \in \R^{\adim}, \; \phi(v) = 1
        \bigr\} = \sup \bigl\{ u \bullet v : v \in \R^{\adim}, \; \phi(v) \leq 1 \bigr\}
    \end{displaymath}
for $ u \in \R^{\adim} $,    and the \emph{$ \phi $-Wulff shape} (of radius $ r $) is the compact and convex set defined as
    \begin{displaymath}
        \wulff^\phi_r
        = \R^{\adim} \cap \bigl\{
        x  : \polar{\phi}(x)  \leq r
        \bigr\}
        = r \wulff^\phi_1, \quad \wulff^\phi_1 = \wulff^\phi. 
    \end{displaymath}
    We set $ \wulff^\phi_0 = \{0\} $. Note that $ \ppolar{\phi} = \phi $ for every norm $ \phi $. Notice that the \emph{Fenchel inequality} follows from the definition of $ \phi^\circ $: 
    $$ u \bullet v \leq \phi^\circ(u)\, \phi(v) \quad \textrm{for $ u, v \in \R^{n+1} $.} $$
    Moreover, it follows from the $ 1 $-homegeneity of $ \phi $ that 
    \begin{equation}\label{eq basic wulff shapes 1}
    	\textrm{$ \phi $ is differentiable at $ u $ if and only if $ \phi $ is differentiable at $ t u $}
    \end{equation}
    whenever $ u \in \R^{n+1}  \setminus \{0\} $ and $ t > 0 $, in which case $ \nabla \phi^\circ(tu) = \nabla \phi^\circ(u) $.
    
    If $ K = \{u : \phi(u) \leq 1\} $ we notice that $ \phi^\circ $ is the support function of $ K $ (cf.\ \cite{Schneiderbook}). From \cite[Corollary 1.7.3 and Theorem 1.7.4]{Schneiderbook} we see that $ \phi^\circ $ is differentiable at a point $ u \in \R^{n+1} \setminus \{0\} $ if and only if  there exists a unique $ v \in K $ so that $ v \bullet u = \phi^\circ(u) $, in which case $ v = \nabla \phi^\circ(u) $ and 
    $$ \Nor(\wulff^\phi, u) = \{t\, \nabla \phi^\circ(u) : t \geq 0\}. $$
    Employing the Fenchel inequality we infer that 
    \begin{equation}\label{eq basic wulff shapes}
    	 \phi(\nabla \phi^\circ(tu)) = \phi(\nabla \phi^\circ (u)) = 1 \quad \textrm{and} \quad u \bullet \nabla \phi^\circ (u)=1
    \end{equation}
    whenever $ u \in \partial \wulff^\phi $  is a point of differentiability of $ \phi^\circ $ and $ t > 0 $. 
     Since by Rademacher theorem $ \phi^\circ $ is differentiable at $ \LM^{n+1} $ a.e.\ $ x \in \R^{n+1} $, we can use \eqref{eq basic wulff shapes 1} to see that $ \phi^\circ $ is differentiable at $ \HM^n $ a.e.\ $ u \in \partial \wulff^\phi $.
\end{miniremark} 

\begin{miniremark}
	Suppose $ \phi $ is a norm on $ \R^{\adim} $. For every $(\HM^{\vdim}, n)$ rectifiable and
	$ \HM^{\vdim} $-measurable subset $ \Sigma \subseteq \R^{\adim} $ we define the
	\emph{$ \phi $-anisotropic area} of $ \Sigma $ by
	\begin{displaymath}
		\mathcal{A}_\phi(\Sigma) = \textint{\Sigma}{} \phi(\nu(x))\ud \HM^{\vdim}(x)  \,,
	\end{displaymath}
	where $ \nu $ is an $ \HM^{\vdim} \restrict \Sigma $-measurable $\sphere{\vdim}$~valued function
	such that $ \nu(x) \in \Nor^{\vdim}(\HM^{\vdim} \restrict \Sigma, x) $ for $ \HM^{\vdim} $ a.e.\
	$ x \in \Sigma $.
	
	If $ E \subseteq \R^{\adim} $ is a set of finite perimeter, we define the \emph{$ \phi $-perimeter} of $ E $ by
	\begin{displaymath}
		\perim{\phi}(E) = \mathcal{A}_\phi(\redbd E) = \textint{\redbd E}{} \phi(\an{E}{}(x)) \ud \HM^{\vdim}(x) \,,
	\end{displaymath}
	where we denote by $ \redbd E $  the reduced boundary and by $ \an{E}{} : \partial^\ast E \rightarrow \mathbb{S}^n $ the measure-theoretic
	exterior normal of $ E $; cf.~\cite[Definition 3.4]{AFP} noting that we employ a different notation with respect to \cite{AFP}.
	
	It follows from \ref{basic wulff shapes} and  the divergence theorem for  sets of finite perimeter  that $$ \an{\wulff^\phi}{}(u) = \frac{\nabla \phi^\circ (u)}{| \nabla \phi^\circ (u) |} \quad \textrm{for $ \HM^n  a.e.\  u \in \partial \wulff^\phi $} $$
and 
\begin{flalign}\label{eq basic wulff shapes 2}
	(n+1)\bigl| \wulff^\phi \bigr| & = \int_{\partial \wulff^\phi} u \bullet \an{\wulff^\phi}{}(u)\, d\HM^n(u) = \int_{\partial \wulff^\phi} \frac{1}{| \nabla \phi^\circ (u) |}\, d\HM^n(u) \notag \\
	& = \int_{\partial \wulff^\phi} \frac{\phi\bigl(\nabla \phi^\circ (u)\bigr)}{| \nabla \phi^\circ (u) |} = \int_{\partial \wulff^\phi} \phi\bigl( \an{\wulff^\phi}{}(u) \bigr)\, d\HM^n(u) = \mathcal{P}_\phi\bigl(\wulff^\phi\bigr).
\end{flalign}
\end{miniremark}

\begin{miniremark}
 Suppose $ \phi $ is a norm on $ \R^{n+1} $.   Given a set $ A \subseteq \R^{\adim} $ we consider the \emph{$ \phi $-distance function}
    \begin{displaymath}
        \da{A}{\phi}(x) = \inf\bigl\{ \polar{\phi}(x-a) : a \in A \bigr\}
        \quad \textrm{for $ x \in \R^{\adim} $} 
    \end{displaymath}
    and the \emph{$ \phi $-nearest point projection}
    \begin{displaymath}
        \anp{A}{\phi} = \overline{A} \cap \bigl\{ a : \polar{\phi}(x-a) = \da{A}{\phi}(x) \bigr\}
        \quad \textrm{for $ x \in \R^{\adim} $}  \,.
    \end{displaymath}

    We notice by triangle inequality that \emph{if
      $ (a, \eta) \in \overline{A} \times \Bdry \wulff^\phi $ and $ t > 0 $ so that
      $ \da{A}{\phi}(a + t \eta) = t $, then $ \da{A}{\phi}(a+s\eta) = s $ whenever
      $ 0 \leq s \leq t $.}  Henceforth, we define the \emph{$ \phi $-unit normal bundle}
    \begin{displaymath}
        N^\phi(A) = \overline{A} \times \Bdry \wulff^\phi \cap \bigl\{
        (a,\eta) : \da{A}{\phi}(a+s\eta) = s \;
        \textrm{for some $ s > 0 $}
        \bigr\}, 
    \end{displaymath}
    the \emph{$ \phi $-reach function} $ \ar{A}{\phi} : N^\phi(A) \rightarrow (0, +\infty] $ by 
    \begin{displaymath}
       r^\phi_A(a, \eta) = \sup \bigl\{ s : \da{A}{\phi}(a + s \eta) = s \}
        \quad \textrm{for $(a, \eta) \in N^\phi(A) $} 
    \end{displaymath}
    and the \emph{$ \phi $-cut locus} of $ A $ by
    \begin{displaymath}
        \Cut^\phi(A) = \bigl\{ a + \ar{A}{\phi}(a, \eta) \eta
        : (a, \eta) \in N^\phi(A), \; \ar{A}{\phi}(a, \eta) < \infty
        \bigr\} \,. 
    \end{displaymath}
\end{miniremark}

\emph{In what follows, if $ \phi $ is the Euclidean norm, then we omit the dependence on $ \phi $ in all
the symbols introduced above.}

\begin{miniremark}\label{distance flow}
    We say that a norm $ \phi $ on $ \R^{\adim} $ is \emph{strictly convex} if for all $ v, w \in \R^{\adim} $ holds that 
    \begin{displaymath}
        \phi(v+w) = \phi(v) + \phi(w) \quad \implies \quad \phi(v) w = \phi(w) v 
    \end{displaymath}
    (or equivalently that $ \phi(u + v) < \phi(u) + \phi(v) $ whenever $ u, v $ are linearly independent).

    Let $ A \subseteq \R^{\adim} $. Strict convexity implies that \emph{if
      $ (a, \eta) \in \overline{A} \times \Bdry \wulff^\phi $ and $ t > 0 $ so that
      $ \da{A}{\phi}(a + t \eta) = t $, then $ \xi^\phi_A(a+s\eta) = \{a\} $ whenever
      $ 0 < s < t $.}  Henceforth, we define the \emph{$ \phi $-distance flow}
    \begin{displaymath}
        F^\phi_A
        : \Gamma^\phi(A) \rightarrow \R^{\adim} \without \overline{A}
    \end{displaymath} by $ F^\phi_A(a, \eta, t) = a +t\eta $ for $(a,\eta, t) \in \Gamma^\phi(A) $,
    where
    \begin{displaymath}
        \Gamma^\phi(A) = \{(a, \eta, t) \in N^\phi(A) \times \R : 0 < t < \ar{A}{\phi}(a, \eta)\} 
    \end{displaymath}
    and we deduce that $ F^\phi_A $ is injective. Moreover, it follows from definitions that
    \begin{displaymath}
        \im F^\phi_A = \R^{\adim} \without (\overline{A} \cup \Cut^\phi(A)). 
    \end{displaymath} 
\end{miniremark}

\begin{miniremark}\label{uniformly ocnvex norms}
    We say that a norm $ \phi $ on $ \R^{\adim} $ is a \emph{$ \cnt{k} $-norm} if
    $ \phi | \R^{\adim} \without \{0\} $ is $ C^k $-regular. If $ \phi $ is a
    $ \cnt{1} $-norm, then it follows from the $ 1 $-homogeneity of $ \phi $ that
    \begin{displaymath}
        \grad \phi(v) \bullet v = \phi(v) \quad \textrm{whenever $ v \in \R^{\adim} \without \{0\} $} 
    \end{displaymath}
    and $ \grad \phi(tv) = \grad \phi(v) $ whenever $ t > 0 $ and
    $ v \in \R^{\adim} \without \{0\} $. In particular, for $ v \in \R^{\adim} $,
    \begin{displaymath}
        \im \Der \grad \phi(v) \subseteq \lin\{v\}^\perp. 
    \end{displaymath}
Moreover, $ \nabla \phi(v) = - \nabla \phi(-v) $ for $ v \in \R^{n+1} \setminus \{0\} $. If $ \phi $ is
a~$\cnt{1} $-norm and $ E \subseteq \R^{n+1} $ is a set of finite perimeter we define its \emph{exterior $\phi$-normal} by
\begin{displaymath}
	\an{E}{\phi} = \grad \phi \circ \an{E}{}.  
\end{displaymath}

    A $ \cnt{2} $-norm $ \phi $ on $ \R^{\adim} $ is \emph{uniformly convex} if there
    exists a constant $ \gamma(\phi) > 0 $ such that
    \begin{displaymath}
	\Der^2 \phi(u)(v,v) \ge \gamma(\phi) |v|^2
	\quad \text{for $ u \in \sphere{n}$ and $ v \in \lin \{ u \}^{\perp} $} \,.
    \end{displaymath}

    Notice that if $ \phi $ is uniformly convex then $ \polar{\phi} $ is also
    a~uniformly convex $ \cnt{2} $-norm. Moreover, $ \Bdry \wulff^\phi $ is a $ \cnt{2} $-hypersurface and we denote by
    $ \bm{n}^\phi : \Bdry \wulff^\phi \rightarrow \sphere{n} $ the exterior unit normal,
    \begin{displaymath}
        \bm{n}^\phi(x) = \frac{\grad \polar{\phi} (x)}{| \grad \polar{\phi}(x) |} \quad \textrm{for $ x \in \Bdry \wulff^\phi $.} 
    \end{displaymath}
    We recall that $ \grad \phi [\R^{\adim} \without \{0\}] = \Bdry \wulff^\phi $
    and $ \grad \phi | \sphere{n} $ is the inverse of $ \bm{n}^\phi $, cf.\ \cite[Lemma
    2.32]{deRosaKolaSantilli}. Notice that $ \bm{n}^\phi(-x) = - \bm{n}^\phi(x) $ for $ x \in \partial \wulff^\phi $.

    If $ A \subseteq \R^{\adim} $ and $ \phi $ is a uniformly convex norm then, cf.~\cite[Remark 5.10]{deRosaKolaSantilli},
    \begin{equation}\label{cut locus}
    	 | \Cut^\phi(A)| =0.
    \end{equation}
    
    We often employ that map
    $ \Phi : \R^{\adim} \times \sphere{\vdim} \rightarrow \R^{\adim} \times \Bdry \wulff^\phi $ given by
    \begin{displaymath}
    	\Phi(a,u) = (a, \grad\phi(u)) \quad \textrm{for $(a,u) \in  \R^{\adim} \times \sphere{\vdim} $}. 
    \end{displaymath}
    This is  a $ \cnt{1} $-diffeomorphism. Moreover, $
   	\Phi \lIm N(A) \rIm = N^\phi(A) $ and $ N^\phi(A) $ is a Borel
   	and countably $ \vdim $-rectifiable subset of $ \overline{A} \times \Bdry \wulff^\phi $;
   	cf.~\cite[Lemma~5.2]{deRosaKolaSantilli}. 
   
\end{miniremark}
    
    \begin{miniremark}\label{curvature smooth hypersurfaces}
Suppose $ \phi $ is a uniformly convex $ \cnt{2} $-norm,    $ \Sigma \subseteq \R^{\adim} $ is a $ \cnt{2} $-hypersurface
    	and $ \nu : \Sigma \rightarrow \sphere{\vdim} $ is a unit-normal $ \cnt{1} $-vector field. Then the linear map
    	$$ \Der (\grad \phi \circ \nu)(a) : \Tan(\Sigma, a) \rightarrow \Tan(\Sigma,a) $$ is
    	diagonalizable (cf.~\cite[Remark~3.13]{HugSantilli}). Moreover, $ N(\Sigma)| \Sigma $ is a $ n $-dimensional
    	$ \cnt{1} $-submanifold of $ \R^{\adim} \times \sphere{\vdim} $ and  $ N^\phi(\Sigma)| \Sigma = \Phi(N(\Sigma)|\Sigma) $ (cf.\ \ref{uniformly ocnvex norms}). In particular, we infer that  \emph{$N^\phi(\Sigma)|\Sigma $ is a $ \vdim $-dimensional $ \cnt{1} $-submanifold of $ \R^{\adim} \times \Bdry \wulff^\phi $.}
    	
    	If $ \Sigma \subseteq \R^{\adim} $ is a $ \cnt{2} $-hypersurface,
    	$(a, \eta) \in N^\phi(\Sigma)| \Sigma $ and $ \nu $ is a unit-normal $ \cnt{1} $-vector field
    	defined in a neighbourhood of $ a $ such that $ \grad \phi(\nu(a)) = \eta $, then we define the
    	\emph{principal $ \phi $-curvatures} of $ \Sigma $ at $ a $ in the direction $ \eta $,
    	\begin{displaymath}
    		\kappa^\phi_{\Sigma, 1}(a, \eta) \leq \ldots \leq \kappa^\phi_{\Sigma, n}(a, \eta),
    	\end{displaymath}
    	to be the eigenvalues of $ \Der (\grad \phi \circ \nu)(a) $. Moreover, if $ a \in \Sigma $ we define the mean $ \phi $-curvature of $ \Sigma $ at $ a $ 
    	$$ \overrightarrow{H}^\phi_\Sigma(a) = \sum_{i=1}^n \kappa^\phi_{\Sigma,i}(a, \nabla \phi(u))\, u $$
    	whenever $ u \in N(\Sigma, a) $. Since for every $ a \in \Sigma $ there exists $ u \in \mathbf{S}^n $ such that $ N(\Sigma,a) = \{u, -u\} $, this definition does not depend on the choice of $ u $, since $ \nabla \phi(u) = - \nabla \phi(-u) $ for every $ u \in \mathbf{S}^n $ and  $ \kappa^\phi_{\Sigma,i}(a, \nabla \phi(u)) = - \kappa^\phi_{\Sigma,i}(a, - \nabla \phi(u)) $.
    \end{miniremark}

\subsection*{First variation of the anisotropic perimeter}

\begin{miniremark}
	Suppose $ \phi $ is a norm on $ \R^{\adim} $. For every $(\HM^{\vdim}, n)$ rectifiable and
	$ \HM^{\vdim} $-measurable subset $ \Sigma \subseteq \R^{\adim} $ we define the
	\emph{$ \phi $-anisotropic area} of $ \Sigma $ by
	\begin{displaymath}
		\mathcal{A}_\phi(\Sigma) = \textint{\Sigma}{} \phi(\nu(x))\ud \HM^{\vdim}(x)  \,,
	\end{displaymath}
	where $ \nu $ is an $ \HM^{\vdim} \restrict \Sigma $-measurable $\sphere{\vdim}$~valued function
	such that $ \nu(x) \in \Nor^{\vdim}(\HM^{\vdim} \restrict \Sigma, x) $ for $ \HM^{\vdim} $ a.e.\
	$ x \in \Sigma $.
	
	If $ E \subseteq \R^{\adim} $ is a set of finite perimeter, we define the \emph{$ \phi $-perimeter} of $ E $ by
	\begin{displaymath}
		\perim{\phi}(E) = \mathcal{A}_\phi(\redbd E) = \textint{\redbd E}{} \phi(\an{E}{}(x)) \ud \HM^{\vdim}(x) \,,
	\end{displaymath}
	where we denote by $ \redbd E $  the reduced boundary and by $ \an{E}{} : \partial^\ast E \rightarrow \mathbb{S}^n $ the measure-theoretic
	exterior normal of $ E $; cf.~\cite[Definition 3.4]{AFP} and notice that our notation differs from \cite{AFP}.  If $ \phi $ is
	a~$\cnt{1} $-norm we define its \emph{$\phi$-normal} by
	\begin{displaymath}
		\an{E}{\phi} = \grad \phi \circ \an{E}{}  
	\end{displaymath}
and we notice that $ \bm{n}^\phi\big(\pm\an{\Omega}{\phi}(a)\big) = \pm \an{\Omega}{}(a) $ for $ a \in \partial^\ast E $.
\end{miniremark}
	
	\begin{miniremark}
		Let $ \VF(\R^{n+1})$  be the space of $ \cnt{1} $-vector field with compact support in $ \R^{n+1} $.
		
Suppose $ \phi $ is a $ \cnt{1} $-norm.	If $ E \subseteq \R^{n+1} $ is a set of
	finite perimeter and $ g \in \VF(\R^{n+1}) $, then \emph{the
		first variation of $ \mathcal{P}_{\phi} $ at $ E $ in the direction $ g $} is defined by
	\begin{displaymath}
		\delta \mathcal{P}_{\phi}(E)(g) = \left. \tfrac{d}{dt} \mathcal{P}_{\phi}(\varphi_t \lIm E \rIm) \right|_{t=0} \,,
	\end{displaymath}
	where $ \varphi_t(x) = x + t g(x) $ for $ (x, t) \in \R^{n+1} \times \R $. We denote by
	$ \| \delta \mathcal{P}_{\phi}(E) \| $ the total variation measure associated with
	$\delta \mathcal{P}_{\phi}(E) $. 
	
For $ \nu \in \sphere{\vdim} $, we define
	$ B_\phi(\nu) \in {\rm Hom}\bigl(\R^{\adim},\R^{\adim}\bigr) $ by the formula
	\begin{displaymath}
		\langle v, B_\phi(\nu) \rangle
		= \phi(\nu) v - \langle v, \Der \phi(\nu) \rangle \, \nu
		\quad \textrm{for $ v \in \R^{\adim} $.}
	\end{displaymath}  
	Notice that $ B_\phi(-\nu) = B_\phi(\nu) $ for each $ u \in \sphere{\vdim} $. It is known (cf.\ \cite[Appendix A]{DePhilippis2018}) that 
	\begin{equation}\label{anisotropic divergence formula}
		\delta \mathcal{P}_\phi(E)(g) = \int_{\partial^\ast E} \Der g(x) \bullet B_\phi(\an{E}{}(x))\, d\HM^n(x) \quad \textrm{for $ g \in \VF(\mathbf{R}^{n+1}) $.} 
	\end{equation} 
	 If $ E \subseteq \R^{\adim}$ is a bounded set of finite perimeter and
	$ X(x) = x $ for $ x \in \R^{\adim} $, then
	\begin{equation}
		\label{anisotropic first variation and anisotorpic perimeter}
		\delta\mathcal{P}_\phi(E)(X) = n \perim{\phi}(E) \,.
	\end{equation}
	Indeed, choosing an orthonormal basis $ e_1, \ldots , e_{n+1} $ of~$ \R^{\adim} $ and noting
	that $ \grad \phi(v) \bullet v = \phi(v) $ for $ v \in \R^{\adim} \without \{0\} $, we compute
	\begin{displaymath}
		B_\phi(\nu) \bullet \id{\R^{\adim}}
		= \textsum{i=1}{n+1} \langle e_i, B_\phi(\nu)\rangle \bullet e_i
		\\
		= (n+1) \phi(\nu) - \langle \nu, \Der \phi(\nu) \rangle = n \phi(\nu)
	\end{displaymath}
	for each $ \nu \in \sphere{\vdim} $, whence \eqref{anisotropic first variation and anisotorpic perimeter} follows from \eqref{anisotropic divergence formula}.
\end{miniremark}

\begin{miniremark}
	Suppose $ \phi $ is a $ \cnt{1} $-norm.	Notice that if $ E \subseteq \R^{n+1} $ is a set of finite perimeter satisfying 
	$$  \| \delta \mathcal{P}_{\phi}(E) \| \leq \kappa \, \bigl( \HM^{\vdim} \restrict \redbd E \bigr) \quad \textrm{for some $ 0 \leq \kappa < \infty $,} $$
	then by Riesz representation theorem and Lebesgue differentiation theorem (cf.\ \cite[2.5.12, 2.9]{Federer1969}) we infer the existence of  a function $ \overrightarrow{H}_E^\phi\in L^\infty(\HM^n \restrict \partial^\ast E, \mathbf{R}^{n+1}) $ such that
	$$ \delta \mathcal{P}_\phi(E)(g) = \int_{\partial^\ast E} g(x) \bullet \overrightarrow{H}^\phi_E(x)\, d\HM^n(x) \quad \textrm{for $ g \in \VF(\R^{n+1}) $.} $$
	The following result summarizes the main properties of a set of finite perimeter $ E $ as above.
	
	\begin{theorem}\label{regularity theorem}
		Suppose $ \phi $ is a uniformly convex $ \cnt{3} $-norm and $ E \subseteq \R^{n+1} $ is a set of finite perimeter such that
 \begin{equation}\label{regularity theorem hp1}
 	\HM^n(\overline{\partial^\ast E} \setminus \partial^\ast E) =0 
 \end{equation}
 and  there exists $ 0 < \kappa_E < \infty $ with 
 \begin{equation}\label{regularity theorem hp2}
 	 \| \delta \mathcal{P}_{\phi}(E) \| \leq \kappa_E \, \bigl( \HM^{\vdim} \restrict \redbd E \bigr). 
 \end{equation}
		
		Then there exists an open subset $ \Omega \subseteq \R^{n+1} $ and a $ \cnt{1, \alpha} $-hypersurface $ M \subseteq \redbd \Omega $ such that the following statements hold.
		\begin{enumerate}
			\item\label{regularity theorem 1} $ \LM^{\adim}(\Omega \without E) = \LM^{\adim}(E \without \Omega) =0 $, $ \HM^n(\Bdry \Omega \setminus \redbd \Omega) =0 $, $ \spt (\HM^n \restrict \redbd \Omega) = \Bdry \Omega $.
			\item\label{regularity theorem 2}	$ \HM^{\vdim}(\Bdry \Omega \without M) = 0$
			\item\label{regularity theorem 5}  There exists a countable family of $ \cnt{2} $-hypersurfaces of $ \R^{n+1}  $ that covers $ \HM^n $ almost all $ \partial \Omega $.
			\item $ N(\partial \Omega,a) = \{\pm \an{\Omega}{}(a)\} $ for $ \HM^n $ a.e.\ $  a\in \partial^\ast \Omega $.
			\item\label{regularity theorem 3}  If $ Z \subseteq \partial \Omega $ and $ \HM^n(Z) =0 $, then $ \HM^n(N^\phi(\partial \Omega)|Z) =0 $.
			\item\label{regularity theorem 4} If $ \Sigma \subseteq \R^{n+1} $ is a  $ \cnt{2} $ hypersurface, then  $$ \overrightarrow{H}^\phi_\Sigma(a) = \overrightarrow{H}^\phi_\Omega(a) = \overrightarrow{H}^\phi_E(a) \quad \text{$ \HM^n $ a.e.\ $ a \in \Sigma \cap \partial \Omega $.} $$
		\end{enumerate}
	\end{theorem}
\begin{proof}
	For $(a)$-$(d) $, see \cite[Corollary 1.3]{KolSan25}. Since the unit-density $ n $-dimensional varifold $ V $ associated with $ \partial \Omega $ satisfies $ \| \delta_\phi V \| \leq \kappa_E \| V \| $, we infer $(e)$ from \cite[Lemma 4.11, Theorem 4.10 and Lemma 5.4]{deRosaKolaSantilli}. Moreover, noting that we also have that $ \HM^n \restrict \spt \| V \| $ is absolutely continuous with respect to $ \| V \| $, we deduce $(f)$ from \cite[Theorem 1.1]{KolSan26}.
\end{proof}
\begin{remark}\label{rmk scalar mean curvature}
Combining \ref{regularity theorem 5} and \ref{regularity theorem 4} we infer that $$ \overrightarrow{H}^\phi_E(a) = \overrightarrow{H}^\phi_\Omega(a) \in \Nor^n(\HM^n \restrict \partial \Omega, a) \quad \textrm{for $  \HM^n $ a.e.\ $ a \in \partial \Omega $.} $$ 
Whenever $ E $ is a set of finite perimeter satisfying the hypothesis of Theorem \ref{regularity theorem} we define the \emph{scalar mean $ \phi $-curvature of $ E $} 
$$  H^\phi_E= \overrightarrow{H}^\phi_E \bullet \an{E}{} $$
and we notice that $ \overrightarrow{H}^\phi_E(a) = H^\phi_E(a)\, \an{E}{}(a) $ for $ \HM^n $ a.e.\ $ a \in \partial^\ast E $.
\end{remark}
	\end{miniremark}

\section{Anisotropic Curvatures of closed sets}\label{sec: anisotorpic curvatures} 
In this section we assume that $ \phi $ is a uniformly convex $ \cnt{2} $-norm on $ \R^{\adim} $. We employ the \emph{$ \phi $-principal curvatures} of a set $ A $, for which we refer to \cite[Section 3]{HugSantilli} for details. They are Borel functions $
        \kappa^\phi_{A,i} : \widetilde{N}^\phi(A) \rightarrow (-\infty, +\infty] $ 
    defined on a Borel subset $ \widetilde{N}^\phi(A) $ of $ N^\phi(A) $, such that
    \begin{displaymath}
    \HM^{\vdim}\bigl(N^\phi(A) \without \widetilde{N}^\phi(A)\bigr) =0
    \end{displaymath}
and 
\begin{displaymath}
	\kappa_{A,1}^\phi(a, \eta) \leq \ldots \leq \kappa^\phi_{A,n}(a, \eta)
\quad \text{for $(a, \eta) \in\widetilde{N}^\phi(A) $} \,.
\end{displaymath}
The principal $ \phi $-curvatures naturally appear in the following two results.

\begin{lemma}[\protect{cf.\, \cite[Lemma 3.9]{HugSantilli}}]\label{principal frame field}
	There exists maps $ \tau_1, \ldots, \tau_n : \widetilde{N}^\phi(A) \rightarrow \mathbf{S}^n $ such that, if we define $ \zeta_1, \ldots, \zeta_n : \widetilde{N}^\phi(A) \rightarrow \R^{n+1} \times \R^{n+1} $ by
	\[
	\zeta_i(a,\eta) = \begin{cases}
	\bigl(\tau_i(a, \eta), \kappa^\phi_{A,i}(a, \eta)\tau_i(a, \eta)\bigr) & \textrm{if $ \kappa^\phi_{A,i}(a, \eta) < \infty $}\\
	\bigl(0, \tau_i(a, \eta)\bigr) & \textrm{if $ \kappa^\phi_{A,i}(a, \eta) = \infty $}\\
	\end{cases}
	\]
	for $ i = 1, \ldots, n $, then 
	$$ \Tan^n(\HM^n \restrict W, (a, \eta)) = \lin\bigl\{\zeta_i(a, \eta) : i = 1, \ldots, n\bigr\} $$
	for every  $ \HM^n $-measurable $ W \subseteq N^\phi(A) $ satisfying $ \HM^n(W) < \infty $.
\end{lemma}

\begin{remark}
	The maps $ \tau_1, \ldots , \tau_n $ can be chosen to be Borel maps. This can be proved adapting the proof of \cite[Lemma 2.48]{KolSan25}. 
\end{remark}

\begin{lemma}[\protect{cf.\, \cite[Corollary 3.18]{HugSantilli}}]
    \label{disintegration}
    \label{lem:disintegration}
    Suppose $ A \subseteq \R^{\adim} $ is a closed set such that
    $ \HM^{\vdim}\bigl(N^\phi(A) \without \widetilde{N}^\phi_n(A)\bigr) =0 $.

    Then there exists a $ \HM^{\vdim} \restrict N^\phi(A) $-measurable function $ J_A^\phi $ such
    that
    \begin{displaymath}
        J_A^\phi(a, \eta) = (\HM^{\vdim} \restrict W,\vdim) \ap  J_n \pp(a, \eta)
        \quad \textrm{for $\HM^{\vdim} $ a.e.\ $(a, \eta) \in W $} 
    \end{displaymath}
    whenever $ W \subseteq N^\phi(A) $ is a set of finite $ \HM^{\vdim} $-measure (hence $ (\HM^{\vdim},n) $-rectifiable). Moreover, 
    \begin{multline}
        \int_{\R^{\adim} \without A} \varphi \ud \LM^{\adim} \\
        = \int_{N^\phi(A)} \phi(\bm{n}^\phi(\eta))\, J_A^\phi(a, \eta)\, \int_0^{r^\phi_A(a, \eta)} \varphi(a + t \eta) \prod_{i=1}^{\vdim}\bigl( 1 + t \kappa^\phi_{A,i}(a, \eta)\bigr)\ud t\ud \HM^{\vdim}(a, \eta)
    \end{multline}
    whenever $ \varphi : \R^{\adim} \rightarrow \R $ is a non-negative Borel function.
\end{lemma}

\begin{lemma}\label{lem approx curvatures}
    Suppose $ A \subseteq \R^{\adim} $ and $ \Sigma \subseteq \R^{\adim} $ is a
    $ \cnt{2} $-hypersurface. Then there exists $ R \subseteq \Sigma \cap \overline{A} $ such that
    \begin{enumerate}
    \item \label{lem approx curvatures conlcusion 1} $ \HM^{\vdim}(\Sigma \cap \overline{A} \without R) =0 $ and $  N(A)|R \subseteq N(\Sigma) $,
    \item \label{lem approx curvatures conlcusion 2} for $ \HM^{\vdim} $ a.e.\ $(a, u) \in N(A)| R $,  the numbers \begin{displaymath} \kappa^\phi_{A,1}(a, \grad \phi(u)), \ldots ,
            \kappa^\phi_{A,n}(a, \grad \phi(u))
        \end{displaymath} are the eigenvalues of $ \Der (\grad \phi \circ \nu)(a) | \Tan(\Sigma, a) $, whenever $ \nu $ is a unit-normal
    $ \cnt{1} $-vector field defined on an~open neighbourhood of $ a $ in $ \Sigma $ such that
$ \nu(a) = u $.

In particular, $ \HM^n\bigl[\bigl(N^\phi(A)| R\bigr) \setminus \widetilde{N}^\phi_n(A)\bigr] =0 $.
    \end{enumerate}
\end{lemma}

\begin{proof}
    Suppose $ R \subseteq \Sigma \cap \overline{A} $ such that $ \Theta^{\vdim}(\HM^{\vdim} \restrict \Sigma \without \overline{A}, a) =0 $ and notice by \cite[2.10.19(4)]{Federer1969} that $ \HM^{\vdim}(\Sigma \cap \overline{A} \without R) =0 $. Employing \cite[3.2.16]{Federer1969} we see that if $ a \in R $ then
    \begin{displaymath}
        \Tan^{\vdim}\big(\HM^{\vdim} \restrict \Sigma \cap \overline{A}, a\big) = \Tan^{\vdim}(\HM^{\vdim} \restrict \Sigma, a) = \Tan(\Sigma, a)  
    \end{displaymath} and 
    \begin{displaymath}
        N(A,a) \subseteq \Nor^{\vdim}\big(\HM^{\vdim} \restrict \Sigma \cap \overline{A}, a\big) \cap \sphere{\vdim}   \subseteq \Nor(\Sigma, a) \cap \sphere{\vdim} = N(\Sigma, a). 
    \end{displaymath}
  Since  $ N^\phi(A)|R = \Phi(N(A)|R)  \subseteq\Phi(N(\Sigma)|R) = N^\phi(\Sigma)| \Sigma $ and $ N^\phi(\Sigma)| \Sigma $ is a $ \vdim $-dimensional $ \cnt{1} $-submanifold, we have  that
    \begin{equation}\label{lem approx curvatures eq1}
        \Tan^{\vdim}(\HM^{\vdim} \restrict N^\phi(A)| R, (a, \eta)) = \Tan(N^\phi(\Sigma)| \Sigma, (a, \eta)) 
    \end{equation}
    for $ \HM^{\vdim} $ a.e.\ $(a, \eta) \in N^\phi(A)| R $. Moreover, if $ \tau_1, \ldots, \tau_n : \widetilde{N}^\phi(A) \rightarrow \mathbf{S}^n $ and $ \zeta_1, \ldots, \zeta_n : \widetilde{N}^\phi(A) \rightarrow \R^{n+1} \times \R^{n+1} $ are given as in Lemma \ref{principal frame field}, we see that
    \begin{equation}\label{lem approx curvatures eq2}
    \Tan^{\vdim}(\HM^{\vdim} \restrict N^\phi(A)| R, (a, \eta))  =\lin\bigl\{ \zeta_i(a, \eta) : i = 1, \ldots, n\bigr\}
    \end{equation}  
for $ \HM^n $ a.e.\ $(a, \eta) \in N^\phi(A)|R $. 

We fix $(a, \eta) \in N^\phi(A)| R $ such that \eqref{lem approx curvatures eq1} and \eqref{lem approx curvatures eq2} hold and define 
$$ \mu_i = \kappa^\phi_{A,i}(a, \eta) \quad \textrm{for $ i = 1, \ldots, n $.} $$
 Let $ V $ be an open neighbourhood of $ a $ in $ \Sigma $ and $ \nu : V \rightarrow \sphere{\vdim} $ is a unit-normal $ \cnt{1} $-vector field such that $ \nabla \phi(\nu(a)) = \eta $. Let  $ G : V \rightarrow V \times \Bdry \wulff^\phi $ be defined as $ G(b) = (b, \grad \phi(\nu(b))) $ for $ b \in V $. Then $ G(V) $ is an open neighbourhood of $ (a, \eta) = G(a) $ in $ N^\phi(\Sigma)| \Sigma $ and 
    \begin{displaymath}
        \Tan (N^\phi(\Sigma)| \Sigma, (a, \eta)) = \Der G(a)[\Tan(\Sigma, a)]. 
    \end{displaymath}
    Let $ \lambda_1 \leq \ldots \leq \lambda_n $ and $ v_1, \ldots, v_n $ be a basis of $ \Tan(\Sigma, a) $ such that 
    \begin{displaymath}
      \langle v_i,  \Der (\grad \phi \circ \nu)(a)\rangle = \lambda_i\, v_i \quad \textrm{for $ i = 1, \ldots, n $.} 
    \end{displaymath}
    Then $ \{(v_i, \lambda_i\, v_i) : i = 1, \ldots, n\} $ is a basis of $ \Tan (N^\phi(\Sigma)| \Sigma, (a, \eta)) $. In particular, 
    $$ \dim \pp\bigl[\Tan (N^\phi(\Sigma)| \Sigma, (a, \eta))\bigr] = n. $$
     Comparing with \eqref{lem approx curvatures eq1} and \eqref{lem approx curvatures eq2} and recalling the formula for $ \zeta_i $ from Lemma \ref{principal frame field}, we deduce that $ \kappa^\phi_{A,n}(a, \eta) < \infty $ and $ \dim {\rm span}\{\tau_1(a, \eta), \ldots, \tau_n(a, \eta)\} = n $. For $ (i,j) \in \{1, \ldots, n\} \times \{1, \ldots, n\} $ let $ a_{ij} \in \mathbf{R} $ be so that
   $$ (\tau_i(a,\eta), \mu_i \,\tau_i(a, \eta)) = \sum_{j=1}^n a_{ij}\, (v_j, \lambda_j\,v_j) \quad \textrm{for $ i = 1, \ldots, n $.} $$
   This implies  for $ i = 1, \ldots, n $ that $$\sum_{j=1}^n a_{ij}\,v_j = \tau_i(a, \eta), \quad \sum_{j=1}^n a_{ij}\,\lambda_j\,v_j = \mu_i\, \tau_i(a, \eta) = \sum_{j=1}^n  \mu_i\, a_{ij}\, v_j, $$
  whence we infer that $ (\mu_i - \lambda_j)\, a_{ij} =0 $
   for every $ i, j \in \{1, \ldots, n\} $. We notice that for every $ i \in \{1, \ldots, n\} $ there must be $ j \in\{1, \ldots, n\} $ so that $ a_{ij} \neq 0 $ and $ \mu_i = \lambda_j $. Henceforth, we conclude that there exists a function
   $ k : \{1,\ldots, n\} \rightarrow \{1, \ldots, n\} $ such that 
   $ \mu_i = \lambda_{k(i)} $ for every $ i \in\{1, \ldots, n\} $. If there was  $ j \in \{1, \ldots, n\} $ so that $ a_{ij} =0 $ for every $ i \in \{1, \ldots, n\} $ then $ \tau_i(a, \eta) \in {\rm span}\{v_h : h \neq j\} $, which contradicts the fact that $ \tau_1(a, \eta), \ldots, \tau_n(a, \eta) $ spans a $ n $-dimensional space. Consequently $ k $ is surjective. Finally, since one easily checks that $ \kappa $ is non-decreasing, we conclude that $ \kappa $ is the identity and we conclude the proof. 
\end{proof}

\begin{theorem}\label{thm locality normal bundle}
   Suppose $ \phi $, $ E $ and $ \Omega $ are as in Theorem \ref{regularity theorem} and $ K = \R^{n+1} \setminus \Omega $. Then
    \begin{equation*}\label{lem properties of admissible domains p1}
   N^\phi(K)| \partial^\ast \Omega = \{(a, -\an{\Omega}{\phi}(a)) : a \in \partial^\ast \Omega \}, \quad    \HM^n\bigl(\partial \Omega \setminus (\partial^\ast \Omega \cap \pp[N(K)])\bigr) =0,
    \end{equation*}
\begin{equation*}
\HM^n \bigl( N^\phi(K) \setminus (N^\phi(K)|\partial^\ast \Omega)\bigr) =0, \quad \HM^n \bigl( N^\phi(K) \setminus \widetilde{N}_n^\phi(K)\bigr) =0
\end{equation*}
    and (cf.\ Remark \ref{rmk scalar mean curvature})
    \begin{equation*}
        \sum_{i=1}^{\vdim} \kappa^\phi_{K,i}\bigl(a, -\an{\Omega}{\phi}(a)\bigr) = - H^\phi_{\Omega}(a)\quad \textrm{for $ \HM^{\vdim} $ a.e.\ $ a \in \Bdry \Omega $.} 
    \end{equation*}
\end{theorem}

\begin{proof}
	 Notice first that if $ a \in \partial^\ast \Omega $ then, by De Giorgi theorem \cite[Theorem 3.59]{AFP}, we see that $ \frac{\Omega- a}{r} $ converges to the halfspace perpendicular to $ \an{\Omega}{} (a) $ and containing $ -\an{\Omega}{}(a) $ as $ r \to 0 $. It is then easy to see that if $ a \in \partial^\ast \Omega $, $ s > 0  $ and $ u \in \mathbb{S}^n $  satisfies  $ \oball{a+su}{s} \subseteq \Omega $, then $ u = - \an{\Omega}{}(a) $; in other words,
	 $$N(K)|  \partial^\ast \Omega = \bigl\{\bigl(a,-\an{\Omega}{}(a)\bigr) : a \in \partial^\ast \Omega \cap \pp[N(K)] \bigr\}. $$
	 Analogously, if $ a \in \partial^\ast \Omega $ is chosen so that $ N(\partial \Omega,a) = \{\pm \an{\Omega}{}(a)\} $, then $ -\an{\Omega}{}(a) \in N(K,a) $. Since, by Theorem \ref{regularity theorem}  we have that $  N(\partial \Omega,a) = \{\pm \an{\Omega}{}(a)\} $ for $ \HM^n $ a.e.\ $ a \in \partial^\ast \Omega $ and $ \HM^n(\partial \Omega \setminus \partial^\ast \Omega) =0 $, we conclude that 
  $$ \HM^n\bigl(\partial \Omega \setminus (\partial^\ast \Omega \cap \pp[N(K)])\bigr) =0. $$

  Let $ \Sigma $ be a $ \cnt{2} $-hypersurface and let $ R \subseteq \Sigma \cap K $ be the set provided by Lemma \ref{lem approx curvatures}  with $ A $ replaced by $ K $. In particular, we notice that $ \HM^n(\Sigma \cap \partial \Omega \setminus R) =0 $. We define $ W $ as the set of points $ (a, u) \in N(K) $ such that $ a \in R \cap \partial^\ast \Omega $ and the conclusion \ref{lem approx curvatures conlcusion 2} of Theorem \ref{lem approx curvatures} holds with $ A $ replaced by $ K $. Observe that
 $$ \HM^n\bigl( N(K) | (R \cap \partial^\ast \Omega) \setminus W \bigr) =0, \quad  \HM^n\bigl( \pp[N(K)] \cap \partial^\ast \Omega \cap R \setminus \pp[W]\bigr) =0, $$
  $$ \HM^n\bigl(\partial \Omega \cap R \setminus \pp[W]\bigr) =0, \quad  \HM^n\bigl(\Sigma \cap \partial \Omega \setminus \pp[W]\bigr)=0.  $$
 We define $$ Q_\Sigma = \bigl\{a \in \pp[W] : \an{\Omega}{}(a) \in \Nor(\Sigma,a)\; \textrm{and}\; \overrightarrow{H}^\phi_\Omega(a) = \overrightarrow{H}^\phi_\Sigma(a) \bigr\} $$
  and, noting that $ \an{\Omega}{}(a) \in \Nor(\Sigma,a) $ for $ \HM^n $ a.e.\ $a \in R \cap \partial^\ast \Omega $ and recalling \ref{regularity theorem 4} and \ref{regularity theorem 3} from Theorem \ref{regularity theorem}  we conclude that 
  $$ \HM^n\bigl(\Sigma \cap \partial \Omega \setminus Q_\Sigma\bigr) =0 \quad \textrm{and} \quad \HM^n\bigl(N^\phi(K)| (\Sigma \setminus Q_\Sigma)\bigr) =0.$$
  If $ a \in Q_\Sigma $, $ u = -\an{\Omega}{}(a) $  and $ \nu $ is a unit-normal vector field on $ \Sigma $ defined on a neighbourhood of $ a $ such that $ \nu(a) = u  $, then $ \kappa_{A,i}^\phi\bigl(a, \nabla \phi(u)\bigr) $ for $ i = 1, \ldots , n $ are the eigenvalues of $ \Der(\nabla \phi \circ \nu)(a) $ and
  $$ \overrightarrow{H}^\phi_\Omega(a) = \overrightarrow{H}^\phi_\Sigma(a) = \sum_{i=1}^n \kappa_{A,i}^\phi\bigl(a, \nabla \phi(u)\bigr) \, u =  - \sum_{i=1}^{\vdim} \kappa^\phi_{K,i}\bigl(a, -\an{\Omega}{\phi}(a)\bigr) \, \an{\Omega}{}(a).$$
  In particular, this implies that
  $$ N^\phi(K)  | Q_\Sigma \subseteq \widetilde{N}^\phi_n(K) \quad \textrm{and} \quad  \HM^n\bigl[ \bigl(N^\phi(K)| \Sigma\bigr) \setminus  \widetilde{N}^\phi_n(K)\bigr] =0. $$

  Since $ \partial \Omega $ is $ \HM^n $ almost all contained in a union of countably many $ \cnt{2} $ hypersurfaces, we readily conclude the proof. 
\end{proof}

\begin{remark}\label{rmk area formula}
	Recalling that $ N^\phi(K) $ is countably $(\HM^n,n) $-rectifiable and Borel, we can find countably many subsets $ W_i \subseteq N^\phi(K)| \partial^\ast \Omega $ such that $ W_i \subseteq W_{i+1} $ and $ \HM^n(W_i) < \infty $ for every $ i \geq 1 $. Henceforth, for every nonnegative Borel function $ g : \R^{n+1} \times \R^{n+1} \rightarrow \R $ we combine monotone convergence theorem, Theorem \ref{disintegration}, coarea formula \cite[3.2.22(3)]{Federer1969} and Theorem \ref{regularity theorem} to infer that 
	\begin{flalign*}
&	\int_{N^\phi(K)}J^\phi_K(a, \eta)\, g(a, \eta)\, d\HM^n(a, \eta) \\
& \qquad = \lim_{i\to \infty}\int_{W_i}\ap J_n \pp(a, \eta)\, g(a, \eta)\, d\HM^n(a, \eta) \\
& \qquad = \lim_{i \to \infty} \int_{\pp(W_i)}\int_{\pp^{-1}(a) \cap W_i}g(a, \eta)d\HM^0(\eta)\, d\HM^n(a) \\
& \qquad = \int_{\pp(N(K))} g(a, -\an{\Omega}{\phi}(a))\, d\HM^n(a) \\
& \qquad = \int_{\partial \Omega} g(a, -\an{\Omega}{\phi}(a))\, d\HM^n(a).
	\end{flalign*}
\end{remark}

\section{Main results}


For a norm $ \phi $ on $ \R^{\adim} $ we define 
\begin{displaymath}
    \gamma_\phi = \sup\{\polar{\phi}(u) : u \in \sphere{\vdim} \}
    \quad \textrm{and} \quad \gamma_\polar{\phi} = \sup \{\phi(u) : u \in \sphere{\vdim}\} \,.
\end{displaymath}
If $ \Omega \subseteq \R^{\adim} $ is an open set we set
\begin{displaymath}
    \delta^\phi_\Omega(x) = \da{\R^{\adim} \without \Omega}{\phi}(x) 
\end{displaymath} and we define 
\begin{displaymath}
 \mathcal{E}^\phi_{\geq r}(\Omega) = \{x \in \R^{\adim}: \delta^\phi_{\Omega}(x) \geq r\} \quad \textrm{for $ r > 0 $.}
\end{displaymath}

\begin{lemma}\label{lemma estimates}
    Suppose  $ n \geq 1 $, $ R > 0 $, $ \lambda > 0 $, $ \gamma > 0 $ and $ \overline{r} = \tfrac{\vdim}{\lambda} $. Then there exists $ C > 0 $ such that if $ \phi $, $ E $ and $ \Omega $ are given as in Theorem \ref{regularity theorem} and satisfy
    $$ \| H^\phi_\Omega - \lambda \|_{L^{\vdim}(\Bdry \Omega)} \leq 1, \quad  \sup\{\gamma^\circ_\phi, \gamma_\phi, \HM^n(\partial \Omega)\} \leq \gamma \quad \textrm{and} \quad   \Omega \subseteq B_R\,, $$ 
    then
    \begin{displaymath}
        \bigg| \big|\mathcal{E}^\phi_{\geq r}(\Omega)\big|- \frac{\big|\Omega\big|}{\overline{r}^{\adim}} (\overline{r} - r)^{\adim} \bigg| \leq C(\overline{r},\gamma, R)\, \| H^\phi_\Omega - \lambda \|_{L^{\vdim}(\Bdry \Omega)}^{\frac{1}{n}} ,
    \end{displaymath}
\begin{displaymath}
	\bigg| \big|\mathcal{E}^\phi_{\geq r}(\Omega)\big|- \frac{\perim{\phi}(\Omega)}{(n+1)\,\overline{r}^{\vdim}} (\overline{r} - r)^{\adim} \bigg| \leq C(\overline{r}, \gamma, R) \| H^\phi_\Omega - \lambda \|_{L^{\vdim}(\Bdry \Omega)}^{\frac{1}{\vdim}} 
\end{displaymath}
and
    \begin{displaymath}
        \bigg| \big|\mathcal{E}^\phi_{\geq r}(\Omega) + \wulff^\phi_s\big| - \frac{\big|\Omega\big|}{\overline{r}^{\adim}} (\overline{r} - (r-s))^{\adim} \bigg| \leq \frac{C(\overline{r},\gamma, R)}{(\overline{r} - r)^{\adim}} \| H^\phi_\Omega - \lambda \|_{L^{\vdim}(\Bdry \Omega)}^{\frac{1}{n}}
    \end{displaymath}
whenever $ 0 < s <  r < \overline{r} $.
\end{lemma}

\begin{proof}
    We define
    \begin{displaymath}
        \Sigma = \{x \in \Bdry \Omega: | H^\phi_\Omega(x)-\lambda | \leq \lambda/2\} 
    \end{displaymath} 
    and notice that 
    \begin{equation}\label{lemma estimates basic estimate 1}
	\HM^{\vdim}(\Bdry \Omega \without \Sigma) \leq \frac{2}{\lambda}\int_{\Bdry \Omega \without \Sigma}| H^\phi_\Omega(x) - \lambda |\ud \HM^{\vdim}(x) \leq \frac{2}{\lambda}\, \,\HM^{\vdim}(\Bdry \Omega)^{\frac{n-1}{\vdim}}\,\| H^\phi_\Omega - \lambda \|_{L^{\vdim}(\Bdry \Omega)}. 
    \end{equation} 
    Moreover, since $ H^\phi_\Omega(x) \geq \frac{\lambda}{2} $ for  $ x \in \Sigma $, we estimate
    \begin{flalign*}
      \frac{\vdim}{n+1}\int_{\Sigma}\frac{\phi(\an{\Omega}{}(x))}{H^\phi_\Omega(x)}\ud \HM^{\vdim}(x) & = \frac{\vdim}{n+1}\int_{\Sigma} \phi(\an{\Omega}{}(x))\, \bigg(\frac{1}{\lambda} + \frac{\lambda - H^\phi_\Omega(x)}{\lambda \, H^\phi_\Omega(x)}\bigg)\ud \HM^{\vdim}(x)\\
                                                                                                 & \leq \frac{n\,\perim{\phi}(\Omega)}{\lambda\,(n+1)}  + \frac{2n\, \gamma^\circ_\phi}{\lambda^2\,(n+1)}\| H^\phi_\Omega - \lambda \|_{L^1(\Bdry \Omega)}.
    \end{flalign*}
    Combining \eqref{anisotropic first variation and anisotorpic perimeter}, Remark \ref{rmk scalar mean curvature} and divergence theorem for sets of finite perimeter we find that
    \begin{flalign*}
      n\,\perim{\phi}(\Omega) & =   \int_{\Bdry \Omega} H^\phi_\Omega(x) \, x \bullet \an{\Omega}{}(x) \ud \HM^{\vdim}(x) \\
                                  &   =    \lambda \int_{\Bdry \Omega}  x \bullet \an{\Omega}{}(x) \ud \HM^{\vdim}(x)  + \int_{\Bdry \Omega} (H^\phi_\Omega(x) -\lambda)\, x \bullet \an{\Omega}{}(x) \ud \HM^{\vdim}(x)\\
                                  & = \lambda \, (n+1)\, | \Omega | + \int_{\Bdry \Omega} (H^\phi_\Omega(x) -\lambda)\, x \bullet \an{\Omega}{}(x) \ud \HM^{\vdim}(x)
    \end{flalign*}
    whence we deduce, since $ \Omega \subseteq B_R $, that
    \begin{equation}\label{lemma estimates basic estimate 18}
	\big| n\,\perim{\phi}(\Omega) - \lambda \, (n+1)\, | \Omega | \big| \leq R\, \| H^\phi_\Omega-\lambda \|_{L^1(\Bdry \Omega)}. 
    \end{equation}
    We conclude that 
    \begin{flalign}\label{lemma estimates basic estimate 2}
      \frac{\vdim}{n+1}\int_{\Sigma}\frac{\phi(\an{\Omega}{}(a))}{H^\phi_\Omega(a)}\ud \HM^{\vdim}(a) & \leq | \Omega | + \bigg( \frac{R}{\lambda\, (n+1)} + \frac{2n\, \gamma^\circ_\phi}{\lambda^2\,(n+1)}\bigg) \| H^\phi_\Omega-\lambda \|_{L^1(\Bdry \Omega)} \notag\\
                                                                                                 & \leq | \Omega | + C_0 \, \| H^\phi_\Omega - \lambda \|_{L^{\vdim}(\Bdry \Omega)}
    \end{flalign}
    where
    \begin{displaymath}
        C_0 =  \bigg( \frac{R}{\lambda\, (n+1)} + \frac{2n\, \gamma^\circ_\phi}{\lambda^2\,(n+1)}\bigg)\, \HM^{\vdim}(\Bdry \Omega)^{\frac{n-1}{\vdim}}. 
    \end{displaymath}

    Now we estimate the volume of $ \Omega  $ using the disintegration formula in \ref{lem:disintegration}. Firstly, recalling Theorem \ref{thm locality normal bundle}, we define  
    \begin{displaymath}
        K = \R^{\adim} \without \Omega, \quad  \tau(a) = r^\phi_K(a, -\an{\Omega}{\phi}(a)) \quad \textrm{for $ \HM^{\vdim} $ a.e.\ $ a \in \Bdry \Omega $}. 
    \end{displaymath}
    Since $ \Omega \subseteq B_R \subseteq R\, \gamma_\phi \, \wulff^\phi $, we conclude that 
    \begin{equation}\label{lemma estimates upper bound for tau 1}
	\tau(a) \leq \gamma_\phi\, R \quad \textrm{for $ \HM^{\vdim} $ a.e.\ $ a \in \Bdry \Omega $.}
    \end{equation}
    Applying the disintegration formula with $ \varphi \equiv 1 $ and $ A = K $ and Remark \ref{rmk area formula}, we compute
    \begin{flalign}\label{lemma estimates volume formula}
      |\Omega| & = \int_{N^\phi(K)}\phi(\bm{\vdim}^\phi(\eta))J_K^\phi(a, \eta)\, \int_0^{r^\phi_K(a, \eta)}\prod_{i=1}^{\vdim}\big( 1 + t \kappa^\phi_{K,i}(a, \eta)\big)\ud t\ud \HM^{\vdim}(a, \eta) \notag\\
               & = \int_{\Bdry \Omega} \phi(\an{\Omega}{}(a)) \, \int_0^{\tau(a)} \prod_{i=1}^{\vdim}\big( 1 + t \kappa^\phi_{K,i}(a, -\an{\Omega}{\phi}(a))\big)\ud t\ud \HM^{\vdim}(a).
    \end{flalign}
    Combining \cite[Remark 3.8]{HugSantilli}, Theorem \ref{thm locality normal bundle} and the arithmetic-geometric mean inequality we infer  that 
    \begin{displaymath}
        -\frac{1}{\tau(a)} \leq \kappa^\phi_{K,i}(a, -\an{\Omega}{\phi}(a)), \quad  1 + t \kappa^\phi_{K,i}(a, -\an{\Omega}{\phi}(a)) > 0 
    \end{displaymath}
    and 
    \begin{equation}\label{lemma estimates basic estimate 13}
	0 < \prod_{i=1}^{\vdim}\big( 1 + t \kappa^\phi_{K,i}(a, -\an{\Omega}{\phi}(a))\big) \leq \bigg( 1 - \frac{t}{\vdim}H^\phi_\Omega(a)\bigg)^{\vdim}  
    \end{equation} 
    for $ i = 1, \ldots , n $, $ 0 < t < \tau(a) $  and for $ \HM^{\vdim} $ a.e.\ $ a \in \Bdry \Omega $. Moreover, since $ H^\phi_\Omega $ is a positive on $ \Sigma $, we also deduce that 
    \begin{equation}\label{lemma estimates upper bound for tau}
	\tau(a) \leq \frac{\vdim}{H^\phi_\Omega(a)} \quad \textrm{for $ \HM^{\vdim} $ a.e.\ $ a \in \Sigma $}. 
    \end{equation} 
    We define
    \begin{displaymath}
        R_1  = \int_{\Bdry \Omega} \phi(\an{\Omega}{}(a)) \, \int_0^{\tau(a)} \bigg[ \bigg( 1 - \frac{t}{\vdim}H^\phi_\Omega(a)\bigg)^{\vdim} - \bigg(\prod_{i=1}^{\vdim}\big( 1 + t \kappa^\phi_{K,i}(a, -\an{\Omega}{\phi}(a))\big)\bigg)\bigg] \ud t\ud \HM^{\vdim}(a)  
    \end{displaymath}
    and 
    \begin{displaymath}
        R_2 := \int_\Sigma \phi(\an{\Omega}{}(a))\int_{\tau(a)}^{\frac{\vdim}{H^\phi_\Omega(a)}}\bigg( 1 - \frac{t}{\vdim}H^\phi_\Omega(a)\bigg)^{\vdim}\ud t\ud \HM^{\vdim}(a). 
    \end{displaymath}
    Employing \eqref{lemma estimates basic estimate 1} and \eqref{lemma estimates upper bound for tau 1} we estimate
    \begin{flalign}
      \label{lemma estimates basic estimate 3}
      &\int_{\Bdry \Omega \without \Sigma} \phi(\an{\Omega}{}(a)) \, \int_0^{\tau(a)}  \bigg( 1 - \frac{t}{\vdim}H^\phi_\Omega(a)\bigg)^{\vdim}  \ud t\ud \HM^{\vdim}(a) \notag  \\
      & \qquad \leq 2^{\vdim}\,\gamma_\phi\, R\, \int_{\Bdry \Omega \without \Sigma}  \phi(\an{\Omega}{}(a))  \bigg[ 1 + \bigg(\frac{\gamma_\phi R}{\vdim}\bigg)^{\vdim}|H^\phi_\Omega(a)|^{\vdim}\bigg]  \ud \HM^{\vdim}(a) \notag \\
      & \qquad \leq 2^{\vdim}\,\gamma^\circ_\phi\,\gamma_\phi\, R\, \int_{\Bdry \Omega \without \Sigma}   \bigg[ 1 + \bigg(\frac{2\gamma_\phi R}{\vdim}\bigg)^{\vdim}\Big(|H^\phi_\Omega(a)-\lambda|^{\vdim} + \lambda^{\vdim}\Big)\bigg]  \ud \HM^{\vdim}(a) \notag \\
      & \qquad \leq C_1\, \| H^\phi_\Omega - \lambda \|_{L^{\vdim}(\Bdry \Omega)} 
    \end{flalign}
    where
    \begin{displaymath}
        C_1 = \frac{2^{\adim}\,\gamma_\phi\, \gamma_\polar{\phi}\, R }{\lambda} \, \Bigg[1 +
        \bigg(\frac{2\gamma_\phi\, R\, \lambda}{\vdim}\bigg)^{\vdim}\Bigg]\, \HM^{\vdim}(\Bdry
        \Omega)^{\frac{n-1}{\vdim}} + 2^{2n}\, \gamma_\phi\, \gamma_\polar{\phi}\, R\,
        \bigg(\frac{\gamma_\phi\, R}{\vdim}\bigg)^{\vdim}. 
    \end{displaymath} Combining \eqref{lemma estimates volume formula},
    \eqref{lemma estimates basic estimate 3} and \eqref{lemma estimates basic estimate 2} we obtain
    \begin{flalign*}
      &| \Omega |  = \int_{\Bdry \Omega} \phi(\an{\Omega}{}(a)) \, \int_0^{\tau(a)}  \bigg( 1 - \frac{t}{\vdim}H^\phi_\Omega(a)\bigg)^{\vdim}  \ud t\ud \HM^{\vdim}(a) - R_1 \\
      & = \int_{\Sigma} \phi(\an{\Omega}{}(a)) \, \int_0^{\frac{\vdim}{H^\phi_\Omega(a)}}  \bigg( 1 - \frac{t}{\vdim}H^\phi_\Omega(a)\bigg)^{\vdim}  \ud t\ud \HM^{\vdim}(a) - R_2 \\
      & \qquad + \int_{\Bdry \Omega \without \Sigma} \phi(\an{\Omega}{}(a)) \, \int_0^{\tau(a)}  \bigg( 1 - \frac{t}{\vdim}H^\phi_\Omega(a)\bigg)^{\vdim}  \ud t\ud \HM^{\vdim}(a) - R_1 \\
      & =	\frac{\vdim}{n+1}\int_{\Sigma}\frac{\phi(\an{\Omega}{}(a))}{H^\phi_\Omega(a)}\ud \HM^{\vdim}(a) - R_2 \\
      & \qquad \qquad \qquad  + \int_{\Bdry \Omega \without \Sigma} \phi(\an{\Omega}{}(a)) \, \int_0^{\tau(a)}  \bigg( 1 - \frac{t}{\vdim}H^\phi_\Omega(a)\bigg)^{\vdim}  \ud t\ud \HM^{\vdim}(a)  - R_1  \\
      & \leq | \Omega | + (C_0 + C_1)\| H^\phi_\Omega - \lambda \|_{L^{\vdim}(\Bdry \Omega)} - R_1 - R_2,
    \end{flalign*}
    in other words, 
    \begin{equation}\label{lemma estimates basic estimate 10}
        R_1 + R_2 \leq (C_0 + C_1)\| H^\phi_\Omega - \lambda \|_{L^{\vdim}(\Bdry \Omega)}.
    \end{equation}

    For $ 0 \leq s < r < \infty $ we define (cf.\ \ref{distance flow})
    \begin{displaymath}
        \Gamma_{s,r} = \{(a, \eta, t) \in N^\phi(K) \times \R :  r^\phi_K(a, \eta) \geq r, \; r-s \leq  t < r^\phi_K(a, \eta)  \} \subseteq \Gamma^\phi(K) 
    \end{displaymath}
    and we prove  that 
    \begin{equation}\label{lemma estimates basic inclusion}
        F_K^\phi(\Gamma_{s,r}) \subseteq \big[\mathcal{E}^\phi_{\geq r}(\Omega) + \wulff^\phi_s \big] \without \Cut^\phi(K) \quad \textrm{for $ 0 < s < r < \infty $}
    \end{equation}
    \begin{equation}\label{lemma estimates equality for eroded sets}
        F_K^\phi(\Gamma_{0,r}) = \mathcal{E}^\phi_{\geq r}(\Omega) \without \Cut^\phi(K) \quad \textrm{for $ 0 < r < \infty $.}
    \end{equation}
    Firstly, recall from \ref{distance flow} that
    $ F^\phi_K(\Gamma^\phi(K)) \cap \Cut^\phi(K) = \varnothing $. Now let $ 0 \leq s < r < \infty $
    and $ (a, \eta, t) \in \Gamma_{s,r} $. If $ t \geq r $ then
    $ \delta^\phi_\Omega(a+t\eta) = t \geq r $ and
    $ a+ t \eta \in \mathcal{E}^\phi_{\geq r}(\Omega) \subseteq \mathcal{E}^\phi_{\geq r}(\Omega) +
    \wulff^\phi_s $; in case $ t < r $, then
    $ a + r \eta \in \mathcal{E}^\phi_{\geq r}(\Omega) $, $ (t-r) \eta \in \wulff^\phi_s $ and
    \begin{displaymath}
        a + t \eta = \big(a +r \eta \big) + (t-r) \eta \in \mathcal{E}^\phi_{\geq r}(\Omega) + \wulff^\phi_s. 
    \end{displaymath}
    This proves \eqref{lemma estimates basic inclusion}. Finally, if
    $ x \in \mathcal{E}^\phi_{\geq r}(\Omega) \without \Cut^\phi(K) $ and
    $ a \in \R^{\adim} \without \Omega $ such that
    \begin{displaymath}
        \polar{\phi} (x-a) = \delta^\phi_\Omega(x) \geq r >0 \,,
    \end{displaymath}
    we define $\eta = \frac{x-a}{\polar{\phi} (x-a)} $ and notice that $(a, \eta) \in N^\phi(K) $;
    additionally, $ x \notin \Cut^\phi(K) $ implies that $ \polar{\phi}(x-a) < r^\phi_K(a, \eta) $,
    whence follows that $(a, \eta, \polar{\phi}(x-a)) \in \Gamma_{0,r} $ and
    $ x = F^\phi_K(a,\eta, \polar{\phi}(x-a)) $. This proves \eqref{lemma estimates equality for
      eroded sets}.

    Since $ F^\phi_K $ is injective (cf.\ \ref{distance flow}), we deduce that for
    $ (a, \eta, t) \in \Gamma^\phi(K) $ the following implication holds for every $ 0 \leq s < r $,
    \begin{displaymath}
        \rchi_{F^\phi_K(\Gamma_{s,r})}(a+t\eta) = 1 \quad \iff \quad (a,\eta, t) \in \Gamma_{s,r}. 
    \end{displaymath}
    Consequently, employing Lemma \ref{disintegration} with $ \varphi =  \rchi_{F^\phi_K(\Gamma_{s,r})} $ and $ A = K $ and Remark \ref{rmk area formula} we deduce that
    \begin{flalign}\label{lemma estimates basic estimate 11}
      &\big| F^\phi_K(\Gamma_{s,r})\big| \notag \\
     & \quad = \int_{N^\phi(K)}
      \phi(\bm{\vdim}^\phi(\eta))\, J^\phi_K(a, \eta)\int_{0}^{r^\phi_K(a, \eta)} \rchi_{F^\phi_K(\Gamma_{s,r})}(a + t\eta)\, \prod_{i=1}^{\vdim}\big( 1 + t \kappa^\phi_{K,i}(a, \eta)\big)
      \ud t \ud \HM^{\vdim}(a, \eta) \notag \\
      & \quad  =  \int_{\Bdry \Omega \cap \{\tau \geq r\}}
      \phi(\an{\Omega}{}(a))\, \int_{r-s}^{\tau(a)} \prod_{i=1}^{\vdim}\big( 1 + t \kappa^\phi_{K,i}(a, -\an{\Omega}{\phi}(a))\big)
      \ud t \ud \HM^{\vdim}(a)
    \end{flalign}
 for $ 0 \leq s < r < \infty $. 

    Suppose $ 0 < r < \overline{r} $.  Noting that \begin{displaymath}
        \R \cap \{t : \max\{r, \tau(a)\} < t < \max\{r, n/H^\phi_\Omega(a)\} \} \subseteq \R \cap \{t :\tau(a) < t <  n/H^\phi_\Omega(a)\}, 
    \end{displaymath} 
and employing \eqref{lemma estimates basic estimate 13}, \eqref{lemma estimates upper bound for tau} and \eqref{lemma estimates basic estimate 10} and \eqref{lemma estimates basic estimate 11}, we estimate
    \begin{flalign}\label{lemma estimates basic estimate 12}
      \big| & F^\phi_K (\Gamma_{0,r})\big| \notag  \\
            &   \geq  \int_{\Sigma \cap \{\tau \geq r\}} \phi(\an{\Omega}{}(a))\, \int_{r}^{\tau (a)} \prod_{i=1}^{\vdim}\big( 1 + t \kappa^\phi_{K,i}(a, -\an{\Omega}{\phi}(a))\big)\ud t\ud \HM^{\vdim}(a) \notag \\
            &  \geq \int_{\Sigma \cap \{\tau \geq r\}} \phi(\an{\Omega}{}(a))\, \int_{r}^{\tau(a)}  \bigg( 1 - \frac{t}{\vdim}H^\phi_\Omega(a)\bigg)^{\vdim} \ud t\ud \HM^{\vdim}(a) - R_1 \notag\\
            &  =   \int_{\Sigma} \phi(\an{\Omega}{}(a))\, \int_{r}^{\max\{\tau(a), r\}}  \bigg( 1 - \frac{t}{\vdim}H^\phi_\Omega(a)\bigg)^{\vdim} \ud t\ud \HM^{\vdim}(a) - R_1 \notag\\
            &  \geq   \int_{\Sigma} \phi(\an{\Omega}{}(a))\, \int_{r}^{\max\{n/H^\phi_\Omega(a), r\}}  \bigg( 1 - \frac{t}{\vdim}H^\phi_\Omega(a)\bigg)^{\vdim} \ud t\ud \HM^{\vdim}(a) - R_1 - R_2 \notag\\
            &   \geq \sum_{\ell =0}^{\vdim}{n \choose \ell}\, \int_{\Sigma} \phi(\an{\Omega}{}(a))\, \int_{r}^{\max\{n/H^\phi_\Omega(a), r\}}  \bigg( 1 - \frac{t\, \lambda}{\vdim}\bigg)^\ell \, \frac{\big(\lambda - H^\phi_\Omega(a) \big)^{n- \ell}\, t^{n- \ell}}{n^{n-\ell}} \ud t  \ud \HM^{\vdim}(a)  \notag \\
            & \qquad   - (C_0 + C_1) \| H^\phi_\Omega - \lambda \|_{L^{\vdim}(\Bdry \Omega)}.
    \end{flalign}
    Recalling that $   H^\phi_\Omega(a) \geq \frac{\lambda}{2} $ for  $ a \in \Sigma $ and noting that 
    \begin{equation}\label{lemma estimates basic estimate 5}
	\Big| 1-\frac{t\,\lambda}{\vdim}\Big| \leq 1 \quad \textrm{for $ 0 \leq t \leq \frac{2n}{\lambda} $,}
    \end{equation}
    we estimate 
    \begin{flalign}\label{lemma estimates basic estimate 6}
      \Bigg| \int_{\Sigma} \phi(\an{\Omega}{}(a))\, & \int_{r}^{\max\{n/H^\phi_\Omega(a), r\}}  \bigg( 1  - \frac{t\, \lambda}{\vdim}\bigg)^\ell \, \frac{\big( \lambda - H^\phi_\Omega(a) \big)^{n- \ell}\, t^{n- \ell}}{n^{n-\ell}}\ud t\ud \HM^{\vdim}(a) \Bigg| \notag \\
        & \leq \int_{\Sigma} \phi(\an{\Omega}{}(a))\, \int_{0}^{\frac{2n}{\lambda}}  \bigg| 1  - \frac{t\, \lambda}{\vdim}\bigg|^\ell \,  \frac{\big| H^\phi_\Omega(a) -\lambda \big|^{n- \ell}\, t^{n- \ell}}{n^{n-\ell}}\ud t\ud \HM^{\vdim}(a)  \notag \\
        & \leq \frac{2^{n-\ell+1}\, n\, \gamma_\polar{\phi}}{\lambda^{n-\ell +1}} \, \int_{\Sigma} \big| H^\phi_\Omega(a)- \lambda \big|^{n-\ell}\ud \HM^{\vdim}(a) \notag \\
       &\leq \frac{2^{n-\ell+1}\, n\, \gamma_\polar{\phi}}{\lambda^{n-\ell +1}} \, \HM^{\vdim}(\Bdry \Omega)^{\frac{\ell}{\vdim}}\, \| H^\phi_\Omega- \lambda \|_{L^{\vdim}(\Bdry \Omega)}^{\frac{n-\ell}{\vdim}} 
    \end{flalign}
    for $ \ell = 0, \ldots , n-1 $.
    We also notice that
    \begin{flalign}\label{lemma estimates basic estimate 7}
      &	\int_{\Sigma} \phi(\an{\Omega}{}(a))\, \int_{r}^{\max\{n/H^\phi_\Omega(a), r\}}  \bigg( 1  - \frac{t\, \lambda}{\vdim}\bigg)^{\vdim} \ud t\ud \HM^{\vdim}(a) \\
      & \quad   = \int_{\Sigma} \phi(\an{\Omega}{}(a))\, \int_{r}^{\overline{r}}  \bigg( 1  - \frac{t\, \lambda}{\vdim}\bigg)^{\vdim} \ud t\ud \HM^{\vdim}(a)  \notag \\
      & \qquad \qquad  +  \int_{\Sigma \cap \{n/H^\phi_\Omega \geq \overline{r}\}} \phi(\an{\Omega}{}(a))\, \int_{\overline{r}}^{n/H^\phi_\Omega(a)}  \bigg( 1  - \frac{t\, \lambda}{\vdim}\bigg)^{\vdim} \ud t\ud \HM^{\vdim}(a) \notag \\
      & \qquad \qquad \qquad  -  \int_{\Sigma \cap \{n/H^\phi_\Omega< \overline{r}\}} \phi(\an{\Omega}{}(a))\, \int_{\max\{r, n/H^\phi_\Omega(a)\}}^{\overline{r}}  \bigg( 1  - \frac{t\, \lambda}{\vdim}\bigg)^{\vdim} \ud t\ud \HM^{\vdim}(a) \notag 
    \end{flalign}
    and, employing \eqref{lemma estimates basic estimate 1} and recalling that  $ \overline{r} = \frac{\vdim}{\lambda} $, we estimate 
    \begin{flalign}\label{lemma estimates basic estimate 19}
      \int_{\Sigma} \phi(\an{\Omega}{}(a))\, \int_{r}^{\overline{r}} &  \bigg( 1  - \frac{t\, \lambda}{\vdim}\bigg)^{\vdim}  \ud t\ud \HM^{\vdim}(a) \notag \\
                                                                                                             & = \frac{(\overline{r} -r)^{\adim}}{(n+1)\overline{r}^{\vdim}}\int_\Sigma \phi(\an{\Omega}{}(a))\ud \HM^{\vdim}(a) \\
                                                                                                             & \geq  \frac{(\overline{r} -r)^{\adim}}{(n+1)\overline{r}^{\vdim}}\, \perim{\phi}(\Omega) - \frac{n\, \gamma_\polar{\phi}}{(n+1)\,\lambda}\HM^{\vdim}(\Bdry \Omega \without \Sigma) \notag\\
                                                                                                             & \geq  \frac{(\overline{r} -r)^{\adim}}{(n+1)\overline{r}^{\vdim}}\, \perim{\phi}(\Omega) - \frac{2n\, \gamma_\polar{\phi}}{(n+1)\,\lambda^2}\HM^{\vdim}(\Bdry \Omega)^{\frac{n-1}{\vdim}}\| H^\phi_\Omega - \lambda \|_{L^{\vdim}(\Bdry \Omega)}. 
    \end{flalign}
    Moreover, using that $ H^\phi_\Omega(a) \geq \frac{\lambda}{2} $ for  $ a \in \Sigma $ and employing H\"older's inequality, we estimate
    \begin{flalign}\label{lemma estimates basic estimate 4}
      \int_{\Sigma \cap \{n/H^\phi_\Omega < \overline{r}\}} \phi(\an{\Omega}{}(a))\, & \int_{\max\{r, n/H^\phi_\Omega(a)\}}^{\overline{r}}  \bigg( 1  - \frac{t\, \lambda}{\vdim}\bigg)^{\vdim} \ud t\ud \HM^{\vdim}(a)  \notag \\
                                                                                     &  \leq \int_{\Sigma \cap \{n/H^\phi_\Omega< \overline{r}\}} \phi(\an{\Omega}{}(a)) \bigg( \overline{r} - \frac{\vdim}{H^\phi_\Omega(a)}\bigg)\ud \HM^{\vdim}(a)  \notag \\
                                                                                     & \qquad \leq \frac{2n\, \gamma_\polar{\phi}}{\lambda^2}\HM^{\vdim}(\Bdry \Omega)^{\frac{n-1}{\vdim}}\, \| H^\phi_\Omega - \lambda \|_{L^{\vdim}(\Bdry \Omega)} 
    \end{flalign}
    and, recalling \eqref{lemma estimates basic estimate 5}, 
    \begin{flalign}\label{lemma estimates basic estimate 8}
      \bigg| \int_{\Sigma \cap  \{n/H^\phi_\Omega \geq \overline{r}\}} &\phi(\an{\Omega}{}(a))\, \int_{\overline{r}}^{n/H^\phi_\Omega(a)}  \bigg( 1  - \frac{t\, \lambda}{\vdim}\bigg)^{\vdim} \ud t\ud \HM^{\vdim}(a) \bigg| \notag \\
      & \leq \int_{\Sigma \cap \{n/H^\phi_\Omega \geq \overline{r}\}} \,\phi(\an{\Omega}{}(a))\, \bigg( \frac{\vdim}{H^\phi_\Omega(a)} - \overline{r}\bigg)\ud \HM^{\vdim}(a) \notag\\
     & \qquad  \leq \frac{2n\,\gamma_\polar{\phi}}{\lambda^2}\, \HM^{\vdim}(\Bdry \Omega)^{\frac{n-1}{\vdim}}\, \| H^\phi_\Omega-\lambda \|_{L^{\vdim}(\Bdry \Omega)}. 
    \end{flalign}

We combine \eqref{cut locus}, \eqref{lemma estimates equality for eroded sets}, \eqref{lemma estimates basic estimate 12}, \eqref{lemma estimates basic estimate 6}, \eqref{lemma estimates basic estimate 19}, \eqref{lemma estimates basic estimate 7}, \eqref{lemma estimates basic estimate 4} and \eqref{lemma estimates basic estimate 8}, and we notice that $$ \| H^\phi_\Omega - \lambda \|^{\frac{n-\ell}{n}}_{L^{\vdim}(\Bdry \Omega)} \leq \| H^\phi_\Omega - \lambda \|^{\frac{1}{n}}_{L^{\vdim}(\Bdry \Omega)} \quad \textrm{$ \ell = 0, \ldots, n-1 $,} $$ to find a positive number $ C_2 $, that depends only on $ n $, $ \gamma_\phi $, $ \gamma^\circ_\phi $, $ R $, $ \lambda $ and $\HM^{\vdim}(\Bdry \Omega) $, such that 
    \begin{equation} \label{lemma estimates basic estimate 9}
        \bigl| \mathcal{E}^\phi_{\geq r}(\Omega) \bigr| = \bigl| F^\phi_K(\Gamma_{0,r}) \bigr| \geq  \frac{(\overline{r} -r)^{\adim}}{(n+1)\overline{r}^{\vdim}}\, \perim{\phi}(\Omega) - C_2\, \| H^\phi_\Omega - \lambda \|_{L^{\vdim}(\Bdry \Omega)}^{\frac{1}{\vdim}} 
    \end{equation}
    for $ 0 < r < \overline{r} $. Moreover, we define $ \Sigma_r = \Sigma \cap \{\tau \geq r\} $ and we employ \eqref{lemma estimates basic estimate 11},  \eqref{lemma estimates upper bound for tau}, \eqref{lemma estimates basic estimate 3}, \eqref{lemma estimates basic estimate 6} and \eqref{lemma estimates basic estimate 8} to find positive constants $ C_3 $ and $ C_4 $ depending on $ n $, $ \gamma_\phi $, $ \gamma^\circ_\phi $, $ R $, $ \lambda $ and $\HM^{\vdim}(\Bdry \Omega) $, such that
    \begin{flalign}\label{lemma estimates basic estimate 14}
      &\bigl| \mathcal{E}^\phi_{\geq r}(\Omega) \bigr| =  \bigl| F^\phi_K(\Gamma_{0,r}) \bigr|  \notag \\
      & \quad \leq \int_{\Sigma_r} \phi(\an{\Omega}{}(a))\, \int_{r}^{\frac{\vdim}{H^\phi_\Omega(a)}}  \bigg( 1 - \frac{t}{\vdim}H^\phi_\Omega(a)\bigg)^{\vdim} \ud t\ud \HM^{\vdim}(a) + C_1 \,  \| H^\phi_\Omega - \lambda \|_{L^{\vdim}(\Bdry \Omega)} \notag \\
      & \quad \leq \int_{\Sigma_r} \phi(\an{\Omega}{}(a))\, \int_{r}^{n/H^\phi_\Omega(a)}  \bigg( 1  - \frac{t\, \lambda}{\vdim}\bigg)^{\vdim} \ud t\ud \HM^{\vdim}(a)  \notag  \\
      & \qquad  +  \sum_{\ell =0}^{n-1} {n \choose \ell} \int_{\Sigma_r} \phi(\an{\Omega}{}(a))\, \int_{0}^{\frac{2n}{\lambda}}  \bigg| 1  - \frac{t\, \lambda}{\vdim}\bigg|^\ell \,  \frac{\big| H^\phi_\Omega(a) -\lambda \big|^{n- \ell}\, t^{n- \ell}}{n^{n-\ell}}\ud t\ud \HM^{\vdim}(a) \notag  \\
      & \qquad \quad  + C_1 \,  \| H^\phi_\Omega - \lambda \|_{L^{\vdim}(\Bdry \Omega)} \notag \\
      & \quad \leq  \int_{\Sigma_r} \phi(\an{\Omega}{}(a))\, \int_{r}^{\overline{r}}  \bigg( 1  - \frac{t\, \lambda}{\vdim}\bigg)^{\vdim} \ud t\ud \HM^{\vdim}(a)  \notag \\
      & \qquad  +  \int_{\Sigma_r \cap \{n/H^\phi_\Omega(a) \geq \overline{r}\}} \phi(\an{\Omega}{}(a))\, \int_{\overline{r}}^{n/H^\phi_\Omega(a)}  \bigg( 1  - \frac{t\, \lambda}{\vdim}\bigg)^{\vdim} \ud t\ud \HM^{\vdim}(a) \notag \\
      & \qquad \quad  -  \int_{\Sigma_r \cap \{n/H^\phi_\Omega(a) < \overline{r}\}} \phi(\an{\Omega}{}(a))\, \int_{ n/H^\phi_\Omega(a)}^{\overline{r}}  \bigg( 1  - \frac{t\, \lambda}{\vdim}\bigg)^{\vdim} \ud t\ud \HM^{\vdim}(a) \notag \\
      & \qquad \qquad  + C_3 \,  \| H^\phi_\Omega - \lambda \|_{L^{\vdim}(\Bdry \Omega)}^{\frac{1}{\vdim}} \notag \\
      & \quad \leq \int_{\Sigma_r} \phi(\an{\Omega}{}(a))\, \int_{r}^{\overline{r}}  \bigg( 1  - \frac{t\, \lambda}{\vdim}\bigg)^{\vdim} \ud t\ud \HM^{\vdim}(a) +  C_4 \,  \| H^\phi_\Omega - \lambda \|_{L^{\vdim}(\Bdry \Omega)}^{\frac{1}{\vdim}}  \notag \\
      & \quad = \frac{(\overline{r} -r)^{\adim}}{(n+1)\overline{r}^{\vdim}}\int_{\Sigma_r} \phi(\an{\Omega}{}(a))\ud \HM^{\vdim}(a)  +  C_4 \,  \| H^\phi_\Omega - \lambda \|_{L^{\vdim}(\Bdry \Omega)}^{\frac{1}{\vdim}}
    \end{flalign}
for $ 0 < r < \overline{r} $. Now we combine \eqref{lemma estimates basic estimate 14} and \eqref{lemma estimates basic estimate 9} to find conclude that 
    \begin{equation}\label{lemma estimates basic estimate 15}
        \int_{\partial \Omega \setminus \Sigma_r} \phi(\an{\Omega}{}(a))\ud \HM^{\vdim}(a)  \leq  \frac{(C_2 + C_4) \, (n+1)\,\overline{r}^{\vdim}}{(\overline{r} -r)^{\adim}}\, \,  \| H^\phi_\Omega - \lambda \|_{L^{\vdim}(\Bdry \Omega)}^{\frac{1}{\vdim}} \quad \textrm{for $ 0 < r < \overline{r} $.}
    \end{equation}

    Suppose now $ 0 < s < r < \overline{r} $. We  prove that there exists a positive constants $ C_5 $ depending on on $ n $, $ \gamma_\phi $, $ \gamma^\circ_\phi $, $ R $, $ \lambda $ and $\HM^{\vdim}(\Bdry \Omega) $ such that  
    \begin{flalign}\label{lemma estimates basic estimate 16}
    	\bigl| \mathcal{E}^\phi_{\geq r}(\Omega) + \wulff^\phi_s\bigr| 
    	 \geq \frac{(\overline{r} -(r-s))^{\adim}}{(n+1)\overline{r}^{\vdim}}\perim{\phi}(\Omega)  - \frac{C_5}{(\overline{r}-r)^{\adim}}\, \| H^\phi_\Omega - \lambda \|_{L^{\vdim}(\Bdry\Omega)}^{\frac{1}{\vdim}}.
    \end{flalign}  
  The proof of \eqref{lemma estimates basic estimate 16} proceeds similarly to \eqref{lemma estimates basic estimate 9},  additionally employing the key estimate \eqref{lemma estimates basic estimate 15}. Firstly, using \eqref{lemma estimates basic inclusion}, \eqref{lemma estimates basic estimate 11},  \eqref{lemma estimates basic estimate 13}, \eqref{lemma estimates upper bound for tau} and \eqref{lemma estimates basic estimate 10} we obtain
    \begin{flalign}
      	\bigl| \mathcal{E}^\phi_{\geq r}(\Omega) + & \wulff^\phi_s\bigr| \\
      & \geq 	\bigl| F^\phi_K(\Gamma_{s,r})\bigr| \\
      &    \geq  \int_{\Sigma_r} \phi(\an{\Omega}{}(a))\, \int_{r-s}^{\tau (a)} \prod_{i=1}^{\vdim}\big( 1 + t \kappa^\phi_{K,i}(a, -\an{\Omega}{\phi}(a))\big)\ud t\ud \HM^{\vdim}(a) \notag \\
      &  \geq \int_{\Sigma_r} \phi(\an{\Omega}{}(a))\, \int_{r-s}^{\tau(a)}  \bigg( 1 - \frac{t}{\vdim}H^\phi_\Omega(a)\bigg)^{\vdim} \ud t\ud \HM^{\vdim}(a) - R_1 \notag\\
      &   \geq \int_{\Sigma_r} \phi(\an{\Omega}{}(a))\, \int_{r-s}^{\frac{n}{H^\phi_\Omega(a)}}  \bigg( 1 - \frac{t}{\vdim}H^\phi_\Omega(a)\bigg)^{\vdim} \ud t\ud \HM^{\vdim}(a) - R_1 - R_2 \notag\\
      & \geq \sum_{\ell =0}^{\vdim}{n \choose \ell}\, \int_{\Sigma_r} \phi(\an{\Omega}{}(a))\, \int_{r-s}^{n/H^\phi_\Omega(a)}  \bigg( 1 - \frac{t\, \lambda}{\vdim}\bigg)^\ell \, \frac{\big(\lambda - H^\phi_\Omega(a) \big)^{n- \ell}\, t^{n- \ell}}{n^{n-\ell}} \ud t \ud \HM^{\vdim}(a)\\
      & \quad   - (C_0 + C_1) \| H^\phi_\Omega - \lambda \|_{L^{\vdim}(\Bdry \Omega)}
    \end{flalign}
and we estimate as in \eqref{lemma estimates basic estimate 6} to find that 
\begin{flalign*}
\bigg|  \int_{\Sigma_r} \phi(\an{\Omega}{}(a))\, \int_{r-s}^{n/H^\phi_\Omega(a)}  \bigg( 1 - \frac{t\, \lambda}{\vdim}\bigg)^\ell \, & \frac{\big(\lambda - H^\phi_\Omega(a) \big)^{n- \ell}\, t^{n- \ell}}{n^{n-\ell}}  \ud t \ud \HM^{\vdim}(a) \bigg| \\
 &\leq \frac{2^{n-\ell+1}\, n\, \gamma_\polar{\phi}}{\lambda^{n-\ell +1}} \, \HM^{\vdim}(\Bdry \Omega)^{\frac{\ell}{\vdim}}\, \| H^\phi_\Omega- \lambda \|_{L^{\vdim}(\Bdry \Omega)}^{\frac{n-\ell}{\vdim}} 
\end{flalign*}
for $ \ell = 0, \ldots , n-1 $. Additionally, we notice that 
     \begin{flalign}\label{lemma estimates basic estimate 7'}
    	&	\int_{\Sigma_r} \phi(\an{\Omega}{}(a))\, \int_{r-s}^{n/H^\phi_\Omega(a)}  \bigg( 1  - \frac{t\, \lambda}{\vdim}\bigg)^{\vdim} \ud t\ud \HM^{\vdim}(a) \\
    	& \quad   = \int_{\Sigma_r} \phi(\an{\Omega}{}(a))\, \int_{r-s}^{\overline{r}}  \bigg( 1  - \frac{t\, \lambda}{\vdim}\bigg)^{\vdim} \ud t\ud \HM^{\vdim}(a)  \notag \\
    	& \qquad \qquad  +  \int_{\Sigma_r \cap \{n/H^\phi_\Omega \geq \overline{r}\}} \phi(\an{\Omega}{}(a))\, \int_{\overline{r}}^{n/H^\phi_\Omega(a)}  \bigg( 1  - \frac{t\, \lambda}{\vdim}\bigg)^{\vdim} \ud t\ud \HM^{\vdim}(a) \notag \\
    	& \qquad \qquad \qquad  -  \int_{\Sigma \cap \{n/H^\phi_\Omega < \overline{r}\}} \phi(\an{\Omega}{}(a))\, \int_{n/H^\phi_\Omega(a)}^{\overline{r}}  \bigg( 1  - \frac{t\, \lambda}{\vdim}\bigg)^{\vdim} \ud t\ud \HM^{\vdim}(a) \notag 
    \end{flalign}
and we apply \eqref{lemma estimates basic estimate 15} to conclude
\begin{flalign*}
  \int_{\Sigma_r} \phi(\an{\Omega}{}(a))\, \int_{r-s}^{\overline{r}} &  \bigg( 1  - \frac{t\, \lambda}{\vdim}\bigg)^{\vdim}  \ud t\ud \HM^{\vdim}(a) \\
& = \frac{(\overline{r} -(r-s))^{\adim}}{(n+1)\overline{r}^{\vdim}}\int_{\Sigma_r} \phi(\an{\Omega}{}(a))\ud \HM^{\vdim}(a) \\
& \geq  \frac{(\overline{r} -(r-s))^{\adim}}{(n+1)\overline{r}^{\vdim}}\, \perim{\phi}(\Omega) - (C_2 + C_4)\,\frac{(\overline{r} - (r-s))^{n+1}}{(\overline{r}-r)^{n+1}}\, \| H^\phi_\Omega - \lambda \|_{L^{\vdim}(\Bdry \Omega)}^{\frac{1}{\vdim}}\\
& \geq  \frac{(\overline{r} -(r-s))^{\adim}}{(n+1)\overline{r}^{\vdim}}\, \perim{\phi}(\Omega) - (C_2 + C_4)\,\frac{\overline{r}^{n+1}}{(\overline{r}-r)^{n+1}}\, \| H^\phi_\Omega - \lambda \|_{L^{\vdim}(\Bdry \Omega)}^{\frac{1}{\vdim}}.
\end{flalign*}
Finally we estimate as in \eqref{lemma estimates basic estimate 4} and \eqref{lemma estimates basic estimate 8} to find that
\begin{flalign*}
	 \int_{\Sigma_r \cap \{n/H^\phi_\Omega < \overline{r}\}} \phi(\an{\Omega}{}(a))\, \int_{n/H^\phi_\Omega(a)}^{\overline{r}}  \bigg( 1  - \frac{t\, \lambda}{\vdim}\bigg)^{\vdim} &\ud t\ud \HM^{\vdim}(a) \\
	 & \leq \frac{2n\, \gamma_\polar{\phi}}{\lambda^2}\HM^{\vdim}(\Bdry \Omega)^{\frac{n-1}{\vdim}}\, \| H^\phi_\Omega - \lambda \|_{L^{\vdim}(\Bdry \Omega)}
	 \end{flalign*}
 and
 \begin{flalign*}
 \bigg| \int_{\Sigma_r \cap \{n/H^\phi_\Omega \geq \overline{r}\}} \phi(\an{\Omega}{}(a))\, \int_{\overline{r}}^{n/H^\phi_\Omega(a)}  \bigg( 1  - \frac{t\, \lambda}{\vdim}\bigg)^{\vdim} & \ud t\ud \HM^{\vdim}(a)  \bigg| \\
 &  \leq \frac{2n\, \gamma_\polar{\phi}}{\lambda^2}\HM^{\vdim}(\Bdry \Omega)^{\frac{n-1}{\vdim}}\, \| H^\phi_\Omega - \lambda \|_{L^{\vdim}(\Bdry \Omega)}.
 \end{flalign*}
We combine the estimates above to obtain \eqref{lemma estimates basic estimate 16}.

    Finally, if $ 0 < s < r < \overline{r} $,  noting that $ \mathcal{E}^\phi_{\geq r}(\Omega) + \wulff^\phi_s \subseteq \mathcal{E}^\phi_{\geq r-s}(\Omega) $, we apply \eqref{lemma estimates basic estimate 14} with $ r $ replaced by $ r-s $ to infer that 
    \begin{flalign}\label{lemma estimates basic estimate 17}
      \bigl| \mathcal{E}^\phi_{\geq r}(\Omega) + \wulff^\phi_s \bigr| & \leq \bigl|\mathcal{E}^\phi_{\geq r-s}(\Omega) \bigr| \\
                                                                           & \leq \frac{(\overline{r} -(r-s))^{\adim}}{(n+1)\overline{r}^{\vdim}}\perim{\phi}(\Omega)  +  C_4 \,  \| H^\phi_\Omega - \lambda \|_{L^{\vdim}(\Bdry \Omega)}^{\frac{1}{\vdim}}.
    \end{flalign} 
    Combining \eqref{lemma estimates basic estimate 9}, \eqref{lemma estimates basic estimate 14}, \eqref{lemma estimates basic estimate 16} and \eqref{lemma estimates basic estimate 17} we conclude that there exists a positive constant $ C_6 $ depending on on $ n $, $ \gamma_\phi $, $ \gamma^\circ_\phi $, $ R $ such that 
    \begin{displaymath}
        \bigg| \big|\mathcal{E}^\phi_{\geq r}(\Omega)\big|- \frac{\perim{\phi}(\Omega)}{(n+1)\,\overline{r}^{\vdim}} (\overline{r} - r)^{\adim} \bigg| \leq C_6 \| H^\phi_\Omega - \lambda \|_{L^{\vdim}(\Bdry \Omega)}^{\frac{1}{\vdim}} 
    \end{displaymath}
    and 
    \begin{displaymath}
        \bigg| \big|\mathcal{E}^\phi_{\geq r}(\Omega) + \wulff^\phi_s\big| - \frac{\perim{\phi}(\Omega)}{(n+1)\,\overline{r}^{\vdim}} (\overline{r} - (r-s))^{\adim} \bigg| \leq \frac{C_6}{(\overline{r} - r)^{\adim}} \| H^\phi_\Omega - \lambda \|_{L^{\vdim}(\Bdry \Omega)}^{\frac{1}{\vdim}}, 
    \end{displaymath}
    whenever $ 0 < s < r < \overline{r} $. 
    
    Finally, we recall \eqref{lemma estimates basic estimate 18} to conclude the proof.
\end{proof}

\begin{remark}\label{rmk uniform estimates on converging norms}
    Suppose $ \{\phi_h  \}_{h \geq 1} $ is a sequence of norms on $ \R^{\adim} $ converging pointwise to a norm $ \phi $ and $ e_1, \ldots , e_{n+1} $ is an orthonormal basis of $ \R^{\adim} $. Since $ \phi_h(v)  \leq | v | \cdot \textsum{i=1}{n+1} \phi_h(e_i) $ for each $ v \in \R^{\adim} $ and $ h \geq 1 $, we define $ C = \sup \bigl\{ \textsum{i=1}{n+1} \phi_h(e_i) : h \geq 1\bigr\} $ we see that
    \begin{displaymath}
       C < \infty \quad \textrm{and} \quad  \phi_h(v)  \leq C | v | \quad \textrm{for each $ v \in \R^{\adim} $ and $ h \geq 1 $}. 
    \end{displaymath}
    Henceforth, $ \Lip (\phi_h) \leq C $ for each $ h \geq 1 $ and $ \phi_h $ converge uniformly on each compact subset of $ \R^{\adim} $ to $ \phi $ (cf.\ \cite[2.10.21]{Federer1969}).
    We deduce that there exist $ 0 <  C' < \infty $ such that 
    \begin{equation*}
        C'| v | \leq \phi_h(v) \leq C | v |\quad \textrm{for $ v \in \R^{\adim} $ and  $ h \geq 1$;} 
    \end{equation*} 
    in particular, if $ E \subseteq \R^{\adim} $ is a set of finite perimeter we obtain
    \begin{displaymath}
        C' \HM^{\vdim}(\redbd E) \leq \perim{\phi_h}(E) \leq C \HM^{\vdim}(\redbd E) \quad \textrm{for every $ h \geq 1 $.} 
    \end{displaymath}
    
    We also observe that the sequence of conjugate norms $ \{\polar{\phi}_h\}_{h \geq 1} $ uniformly converges on each compact subset of $ \R^{\adim} $ to $ \polar{\phi} $. Clearly, it is enough to check that $ \phi_h^\circ \to \polar{\phi} $ pointwise on $ \R^{\adim} $. The latter holds since whenever $ f : \R^{\adim} \rightarrow \R $ is a continuous function, the uniform converge on compact subsets of $ \phi_h $ to $ \phi $ implies that for every $ \epsilon > 0 $ there exists $ h_\epsilon \geq 1 $ so that 
    $$ \wulff_{1-\epsilon}^{\phi^\circ} \subseteq \wulff^{\phi_h^\circ}_1 \subseteq \wulff^{\phi^\circ}_{1+\epsilon} \quad \textrm{for $ h \geq h_\epsilon $} $$
    and 
 \begin{displaymath}
        {\textstyle \sup_{\wulff^{\phi^\circ}_{1-\epsilon}} f } \leq \liminf_{h \to \infty} \big({\textstyle\sup_{\wulff^{\phi_h^\circ}_{1}} f}\big) \leq \limsup_{h \to \infty} \big({\textstyle\sup_{\wulff^{\phi_h^\circ}_{1}} f} \big)\leq {\textstyle\sup_{\wulff^{\phi^\circ}_{1+\epsilon}} f}. 
    \end{displaymath}
   Letting $ \epsilon \to 0 $ we see that $$ \lim_{h \to \infty} {\textstyle\sup_{\wulff^{\phi_h^\circ}_{1}} f = \sup_{\wulff^{\phi^\circ}_{1}} f}. $$
 In particular, we conclude that there exists $ 0 < C_1 \leq C_2 < \infty $ such that 
    \begin{displaymath}
        C_1| v | \leq \polar{\phi}_h(v) \leq C_2 | v |\quad \textrm{for $ v \in \R^{\adim} $ and  $ h \geq 1$.} 
    \end{displaymath}
\end{remark}

We are now ready to prove the compactness theorem.

\begin{theorem}\label{thm compactness}
    Suppose $ \phi $ is an arbitrary norm of $ \R^{n+1} $ and $\{\phi_h\}_{h \geq 1} $ is a sequence of uniformly convex $ \cnt{3} $-norms of $ \R^{\adim} $ pointwise converging to $ \phi $. Suppose $ \{E_h\}_{h \geq 1}  $ is a sequence of sets of finite perimeter in $ \R^{n+1} $ such that
    \begin{equation*}
     \HM^n(\overline{\partial^\ast E_h} \setminus \partial^\ast E_h) =0\,, \qquad 	\sup_{h \geq 1} \bigl( \diam (E_h) + \HM^{\vdim}(\partial^\ast E_h)\bigr) < \infty\,, 
    \end{equation*}
     there exists $ 0 < \kappa_h < \infty $  with 
    \begin{equation*}
    	\| \delta \perim{\phi_h}(E_h) \| \leq \kappa_h \, \bigl( \HM^{\vdim} \restrict \redbd E_h \bigr)
    \end{equation*}
and
\begin{equation*}
        \| H^{\phi_h}_{E_h} - \lambda \|_{L^{\vdim}(\partial^\ast E_h)} \rightarrow 0 \quad \textrm{for some $ \lambda > 0 $ as $ h \to \infty $.} 
    \end{equation*}

  Set $ \overline{r} = \tfrac{n}{\lambda} $. Then there exists $ C \subseteq \R^{\adim} $ with $ 0 \leq {\rm card}(C) < \infty $ and $ \polar{\phi}(c-d) \geq \frac{2n}{\lambda} $ for every $ c ,d \in C $, such that if $ E = \bigcup_{c \in C} \big(c + \wulff^\phi_{\overline{r}}\big)  $ then, up to  translations and up to extracting subsequences,
    \begin{displaymath}
 \lim_{h \to \infty}     | E_{h} \symdiff E| =0 \quad \textrm{and} \quad \lim_{h \to \infty}\mathcal{P}_{\phi_h}(E_h) = {\rm card}(C)\, \mathcal{P}_\phi(\wulff^\phi_{\overline{r}}).
    \end{displaymath}
\end{theorem}

\begin{proof}
By Theorem \ref{regularity theorem} we can replace each set $ E_h $ with an open set $ \Omega_h $ satisfying the conclusion \ref{regularity theorem 1}-\ref{regularity theorem 4} of Theorem \ref{regularity theorem}. If $ R = \sup_{h \geq 1} \diam(\Omega_h) < \infty $ and $ p_h \in \Omega_h $ for each $ h \geq 1 $, then $ \Omega_h - p_h \subseteq B_R $  for every $ h \geq 1 $. Henceforth  we can assume, up to translations, that 
$$ \Omega_h \subseteq B_R \quad \textrm{for each $ h \geq 1 $.} $$
 and we deduce from compactness theorem for sets of finite perimeter that there exists a set of finite perimeter $ \Omega \subseteq B_R $ such that, up to subsequences, 
    \begin{displaymath}
        | \Omega_h \symdiff \Omega | \to 0 \quad \textrm{as $ h \to \infty $.} 
    \end{displaymath}
By Remark \ref{rmk uniform estimates on converging norms}, $ \phi_h^\circ $ uniformly converges to $ \phi^\circ $ on compact sets as $ h \to \infty $,
$$ \gamma : = \sup_{h \geq 1}\,\sup\,\bigr\{\gamma_{\phi_h}^\circ, \gamma_{\phi_h}, \HM^n(\partial \Omega_h)\bigl\} < \infty\,, $$
 and there exist $ 0 < C_1 \leq C_2 < \infty $ and $ \gamma > 0 $ such that 
\begin{equation}\label{thm compactness volume 4}
	C_1| v | \leq \polar{\phi}_h(v) \leq C_2 | v |\quad \textrm{for $ v \in \R^{\adim} $ and  $ h \geq 1$} 
\end{equation}
We deduce that
    \begin{displaymath}
        \sup_{p \in \R^{\adim}} \delta^{\phi_h}_{\Omega_h}(p)  \leq 2 C_2 R \quad \textrm{and}  \quad \Lip \big(\delta^{\phi_h}_{\Omega_h}\big) \leq C_2   \quad \textrm{for $ h \geq 1 $.}
    \end{displaymath}
    Henceforth, we can apply Ascoli-Arzela theorem (cf.\ \cite[2.10.21]{Federer1969}) to conclude that there exists a nonnegative function $ f : \R^{\adim} \rightarrow \R $ with $ \Lip(f) \leq C_2 $ such that, up to subsequences, $ \delta^{\phi_{h}}_{\Omega_{h}} $ converges uniformly on each compact subset of $ \R^{\adim} $ to $ f $.  

    We define $ P = \{f >0\} $  and we notice that $ P \subseteq B_R $.  We claim that $ f = \delta^\phi_P $ (we do not exclude that $ P $ might be empty, in which case $ \delta^\phi_P = \dist^\phi_{\R^{n+1}} $ is identically zero). We fix  $ x \in \R^{\adim} $ and  we choose $ a_h \in \R^{\adim} \without \Omega_{h} $ such that $ \polar{\phi}_h(x-a_h) = \delta^{\phi_{h}}_{\Omega_{h}}(x)  $. Since  the sequence $ \{a_h\}_{h \geq 1} $ is bounded, there exists $ a \in \R^{\adim} $ such that $ a_h \to a $ up to subsequences. Moreover, noting that $ \delta^{\phi_{h}}_{\Omega_{h}}(a_h) =0 $ for each $ h \geq 1 $, we have that
    \begin{displaymath}
        f(a_h) = f(a_h) - \delta^{\phi_{h}}_{\Omega_{h}} (a_h) \to 0  
    \end{displaymath}
    hence $ f(a) =0 $ and $  a\in \R^{\adim} \without P $. We deduce from the uniform converge on compact subsets of $\phi_h^\circ $ to $ \phi^\circ $ that
    \begin{displaymath}
        \delta^\phi_P(x) \leq \polar{\phi}(x-a) = \lim_{h \to \infty}\phi_h^\circ(x-a_h) = \lim_{h \to \infty}\delta^{\phi_{h}}_{\Omega_{h}} (x) = f(x). 
    \end{displaymath}
    To prove the opposite inequality we choose $ x \in \R^{\adim} $ and $ a \in \R^{\adim} \without P $ so that $ \polar{\phi}(x-a) = \delta^\phi_P(x) $. For each $ h \geq 1 $ we choose $ a_h \in \R^{\adim} \without \Omega_h $ so that $ \phi_h^\circ(a_h-a) = \delta^{\phi_{h}}_{\Omega_{h}}(a) $ and we notice that $ \polar{\phi}_h(a_h -a) \to f(a) =0 $. It follows from \eqref{thm compactness volume 4} that $ a_h \to a $ and $ \phi_h^\circ(x-a_h) \to \polar{\phi}(x-a) $ as $ h \to \infty $. Noting that $ \delta^{\phi_{h}}_{\Omega_{h}}(x) \leq \phi_h^\circ(x-a_h) $ for each $ h \geq 1 $, we conclude that $ f(x) \leq \delta^\phi_P(x) $.

    Fix $ \overline{r} = \frac{\vdim}{\lambda} $ and  $ 0 < s < r < \overline{r} $. We choose $ \epsilon $ so that $$ 0 <  s - \epsilon, \quad  s + \epsilon < r -\epsilon, \quad r+\epsilon < \overline{r} $$  and we find $ h(\epsilon) \geq 1 $ so that 
    \begin{displaymath}
        \| \delta^{\phi_h}_{\Omega_h} - \delta^\phi_P \|_{L^\infty(B_R)} \leq \epsilon \quad \textrm{and} \quad \| \phi_h^\circ - \polar{\phi} \|_{L^\infty(B_{r/C_1})} \leq \epsilon \quad \textrm{for every $ h \geq h(\epsilon) $.}
    \end{displaymath}  
If necessary we choose $ h(\epsilon) $ larger so that $ \| H^{\phi_h}_{E_h} - \lambda \|_{L^{\vdim}(\partial^\ast E_h)} \leq 1 $ for every $ h \geq h(\epsilon) $.
Since $ \wulff^{\phi_h}_t \cup \wulff^\phi_t \subseteq B_{t/C_1} $ for $ t > 0 $ and $ h \geq 1 $, it follows that
    \begin{equation*}
	 \mathcal{E}^{\phi_h}_{\geq r+\epsilon}(\Omega_h) \subseteq \mathcal{E}^\phi_{\geq r}(P) \subseteq \mathcal{E}^{\phi_h}_{\geq r-\epsilon}(\Omega_h)  
    \end{equation*} 
    \begin{equation*}
   \mathcal{E}^{\phi_h}_{\geq r +\epsilon}(\Omega_h) + \wulff^{\phi_h}_{s - \epsilon} \subseteq      \mathcal{E}^\phi_{\geq r}(P) + \wulff^{\phi}_{s} \subseteq \mathcal{E}^{\phi_h}_{\geq r-\epsilon}(\Omega_h) + \wulff^{\phi_h}_{s + \epsilon}
    \end{equation*}
    for every $ h \geq h(\epsilon) $. Since $ 	\mathcal{E}^\phi_{\geq r}(P)  \without \Omega \subseteq \mathcal{E}^{\phi_h}_{\geq r-\epsilon}(\Omega_h) \without \Omega \subseteq \Omega_h \without \Omega $ for every $ h \geq h(\epsilon) $, we see that \begin{equation}\label{thm compactness volume 5}
         \big| \mathcal{E}^\phi_{\geq r}(P)  \without \Omega \big| =0.
    \end{equation} 
Define $ \delta_h = \| H^{\phi_h}_{\Omega_h} - \lambda \|_{L^{\vdim}(\Bdry \Omega_h)}^{\frac{1}{n}} $ and we employ Lemma \ref{lemma estimates} to find a constant $ C > 0 $ depending on $ \overline{r} $, $ R $ and $ \gamma $ such that
    \begin{displaymath}
        \frac{\big|\Omega_h\big|}{\overline{r}^{\adim}} (\overline{r} - (r + \epsilon))^{\adim}  - C \delta_h \leq \big| \mathcal{E}^\phi_{\geq r}(P) \big| \leq  \frac{\big|\Omega_h\big|}{\overline{r}^{\adim}} (\overline{r} - (r- \epsilon))^{\adim}  + C \delta_h 
    \end{displaymath}
 \begin{displaymath}
	\frac{\mathcal{P}_{\phi_h}(\Omega_h)}{(n+1)\overline{r}^{n}}\, (\overline{r} - (r + \epsilon))^{\adim}  - C \delta_h \leq \big| \mathcal{E}^\phi_{\geq r}(P) \big| \leq  \frac{\mathcal{P}_{\phi_h}(\Omega_h)}{(n+1)\overline{r}^{n}}\, (\overline{r} - (r - \epsilon))^{\adim}  + C \delta_h 
\end{displaymath}
    and 
    \begin{flalign*}
      \frac{\big|\Omega_h\big|}{\overline{r}^{\adim}} (\overline{r} - (r-s+2\epsilon) )^{\adim} - & \frac{C\delta_h}{(\overline{r}  - (r+\epsilon))^{\adim}} \\
       & \leq \big| \mathcal{E}^{\phi_h}_{\geq r+\epsilon}(\Omega_h) + \wulff^{\phi_h}_{s-\epsilon} \big| \\
      & \leq \big| \mathcal{E}^\phi_{\geq r}(P) + \wulff^{\phi}_{s} \big| \\
       & \leq \big| \mathcal{E}^{\phi_h}_{\geq r-\epsilon}(\Omega_h) + \wulff^{\phi_h}_{s+\epsilon} \big| \\
       & \leq \frac{\big|\Omega_h\big|}{\overline{r}^{\adim}} (\overline{r} - (r-s-2\epsilon) )^{\adim}  + \frac{C\delta_h}{(\overline{r} - (r-\epsilon))^{\adim}} 
    \end{flalign*}
    for every $ h \geq h(\epsilon) $. Letting $ h \to \infty $ and then $ \epsilon \to 0 $ we obtain 
    \begin{equation}\label{thm compactness volume 1}
\frac{(\overline{r}-r)^{n+1}}{(n+1)\, \overline{r}^n}\, \bigl( \lim_{h \to \infty}\mathcal{P}_{\phi_h}(\Omega_h)\bigr)  =  \big| \mathcal{E}^\phi_{\geq r}(P) \big| =  \frac{\big|\Omega\big|}{\overline{r}^{\adim}} (\overline{r} - r)^{\adim} 
    \end{equation}  
    and
    \begin{equation}\label{thm compactness volume 2}
        \big| \mathcal{E}^\phi_{\geq r}(P) + \wulff^{\phi}_{s} \big| =   \frac{\big|\Omega\big|}{\overline{r}^{\adim}} (\overline{r} - (r-s) )^{\adim} 
    \end{equation} 
    for every $ 0 < s < r < \overline{r} $.
    
     Letting $ r \to 0 $ in \eqref{thm compactness volume 5} and \eqref{thm compactness volume 1}  we find that 
     \begin{equation}\label{thm compactness volume 8}
| P \setminus \Omega | =0 \quad \textrm{and}    \quad   \frac{\overline{r}}{(n+1)}\, \lim_{h \to \infty}\mathcal{P}_{\phi_h}(\Omega_h)  = | P | = | \Omega |.
     \end{equation}
 In particular, $ | P \triangle \Omega | =0 $.
Moreover, letting $ r \to \overline{r} $ in \eqref{thm compactness volume 2} we obtain 
    \begin{equation}\label{thm compactness volume 3}
        \big| \mathcal{E}^\phi_{\geq \overline{r}}(P) + \wulff^\phi_{s}\big| = \frac{|P|}{\overline{r}^{\adim}}s^{\adim} \quad \textrm{for $ 0 < s < \overline{r}$}
    \end{equation}
and letting $ s \to \overline{r} $ in \eqref{thm compactness volume 3} we deduce that 
\begin{equation}\label{thm compactness volume 6}
	\big| \mathcal{E}^\phi_{\geq \overline{r}}(P) + {\rm int}\,\wulff^\phi_{\overline{r}}\big| = |P|.
\end{equation}
   
    Now we choose a positive integer $ N \geq 1 $ such that \begin{equation}\label{thm compactness volume 7}
        (N-1) \leq \frac{| P |}{\overline{r}^{\adim}| \wulff^\phi_1 |} < N
    \end{equation}
    we claim that 
    \begin{center}
     \emph{$ \mathcal{E}^\phi_{\geq \overline{r}}(P) $ contains precisely $ N-1 $ points.} 
    \end{center}
   To prove the claim we observe that if $ B $ is an open $ \polar{\phi} $-ball of radius $ \overline{r} $, then $ B \cap \mathcal{E}^\phi_{\geq \overline{r}}(P) $ contains at most $ N-1 $ points:  indeed, if $ x_1, \ldots , x_k \in B \cap \mathcal{E}^\phi_{\geq \overline{r}}(P) $ such that $ \polar{\phi}(x_i - x_j) > 0 $ for $ i \neq j $, then we choose $ s $ so that
    \begin{displaymath}
        0 < s < \inf\bigg\{\frac{\polar{\phi}(x_i - x_j)}{2} : i, j = 1, \ldots , k, \; i \neq j  \bigg\} 
    \end{displaymath}
    and, noting that $ s < \overline{r} $, we infer that 
    \begin{displaymath}
        k \big| \wulff^\phi_1\big| s^{\adim} = \bigg| \bigcup_{i=1}^{k} x_i + \wulff^\phi_s \bigg| \leq \big| \mathcal{E}^\phi_{\geq \overline{r}}(P) + \wulff^\phi_{s} \big| = \frac{|P|}{\overline{r}^{\adim}}s^{\adim},  
    \end{displaymath}
    which means that $ k \leq \frac{| P |}{\overline{r}^{\adim}| \wulff^\phi_1 |} < N $. In particular $ \mathcal{E}^\phi_{\geq \overline{r}}(P) $ contains finitely many points and, choosing $ s < \overline{r} $ so that 
    \begin{displaymath}
        0 < s < \inf\bigg\{\frac{\polar{\phi}(x- y)}{2} : x, y \in \mathcal{E}^\phi_{\geq \overline{r}}(P), \; x \neq y \bigg\} 
    \end{displaymath}
    we infer from \eqref{thm compactness volume 3} that 
    \begin{displaymath}
        \frac{|P|}{\overline{r}^{\adim}}s^{\adim} = \big| \mathcal{E}^\phi_{\geq \overline{r}}(P) + \wulff^\phi_s\big| =  {\rm card}\big(\mathcal{E}^\phi_{\geq \overline{r}}(P) \big)\, | \wulff^\phi_1 |\, s^{\adim} 
    \end{displaymath}
    and we conclude that 
    \begin{displaymath}
        {\rm card}\big(\mathcal{E}^\phi_{\geq \overline{r}}(P) \big) = \frac{|P|}{\overline{r}^{\adim}\, \big| \wulff^\phi_1 \big|}  \geq N-1.
    \end{displaymath}
Henceforth, $  {\rm card}\big(\mathcal{E}^\phi_{\geq \overline{r}}(P) \big) = N-1 $. 

Define $$ C = \mathcal{E}^\phi_{\geq \overline{r}}(P) \quad \textrm{and} \quad  E = \mathcal{E}^\phi_{\geq \overline{r}}(P) + {\rm int}\,\wulff^\phi_{\overline{r}} = \bigcup_{x \in C} \bigl( x + {\rm int}\, \wulff^\phi_{\overline{r}}\bigr). $$
By \eqref{thm compactness volume 6}  and \eqref{thm compactness volume 7} we infer that
$$ (N-1)\bigl| \wulff^\phi_{\overline{r}} \bigr| \leq | P | = | E |  \leq  (N-1)\bigl| \wulff^\phi_{\overline{r}} \bigr|, $$
henceforth $ ( x + {\rm int}\, \wulff^\phi_{\overline{r}}) \cap ( y + {\rm int}\, \wulff^\phi_{\overline{r}}) = \varnothing $ whenever $ x , y \in \mathcal{E}^\phi_{\geq \overline{r}}(P) $ and $ x \neq y $. Since $ E \subseteq P $ we see that $ | E \symdiff P | =0 $ and consequently, $ | \Omega_h \symdiff E | \to 0 $ as $ h \to \infty $. Finally, employing \eqref{thm compactness volume 8}  and \eqref{eq basic wulff shapes 2} we see that 
\begin{flalign*}
\lim_{h \to \infty}\mathcal{P}_{\phi_h}(\Omega_h) =  (n+1)\, \overline{r}^n\, {\rm card}(C)\,| \wulff^\phi_1 | = {\rm card}(C)\,\mathcal{P}_\phi\bigl(\wulff^\phi_{\overline{r}}\bigr).
\end{flalign*}
\end{proof}

\begin{remark}\label{rmk compactness 1}
Since Theorem \ref{thm compactness} deals with sequences of sets that are not assumed to be smooth, this result is new even if $ \phi_h = \phi = \text{Euclidean norm} $. We remark that if $ \phi $ is the Euclidean norm then \eqref{regularity theorem hp1} of Theorem \ref{regularity theorem} follows from \eqref{regularity theorem hp2}, as one can see using \cite[8.6]{Allard1972}.  
\end{remark}

\begin{remark}\label{rmk compactness 2}
	Suppose $ \phi $ is a uniformly convex $ \cnt{3} $-norm. If $ \phi_h = \phi $ for every $ h \geq 1 $, then we notice that \cite[Corollary 1.2]{deRosaKolaSantilli} is a special case of Theorem \ref{thm compactness}.
\end{remark}

\begin{remark}\label{rmk compactness 3}
Prior to our result, bubbling phenomena for approximating sequences of sets of finite perimeter were studied in \cite[Corollary 1.2]{delgadinomaggi} in the Euclidean setting, and in \cite{MaggiSantilli} in a general class of warped product spaces, under the additional assumption that the perimeter is preserved in the limit. This hypothesis is essential in these results, as they are deduced as corollaries of uniqueness theorems for volume-constrained critical points of perimeter functionals. On the other hand, Theorem \ref{thm compactness} do not assume preservation of perimeter in the limit, which is a key novelty in this direction.
\end{remark}

{\small\subsection*{Acknowledgements}
 This article is published with funding from Università degli Studi dell'Aquila under the Call for Proposals for Fundamental research and Early-career research grants - Year 2026.
 This research has been partially supported by INDAM-GNSAGA and  PRIN project 20225J97H5.}


\bigskip

{\small \noindent
  Mario Santilli \\
  Department of Information Engineering, Computer Science and Mathematics,\\
  Università degli Studi dell'Aquila\\
  via Vetoio 1, 67100 L’Aquila, Italy\\
  \texttt{mario.santilli@univaq.it}
}
\end{document}